\documentclass[12pt]{article}
\usepackage{amsmath, amsfonts, amssymb,amsthm}
\usepackage{epsfig}
\usepackage{graphics}
\usepackage{pifont}
\usepackage{verbatim}

\newcommand{\Diff}{{\mathrm {Diff}}}

\newcommand{\R}{\mathbb{R}}

\newcommand{\Z}{\mathbb{Z}}
\newcommand{\HH}{\mathbb{H}}
\newcommand{\eps}{\varepsilon}
\newcommand{\bv}{{\bf{v}}}
\newcommand{\bw}{{\bf{w}}}
\newcommand{\bt}{{\bf{t}}}

\newcommand{\id}{\mathrm{Id}}
\newcommand {\p}{\partial}
\newcommand {\wt}{\widetilde}
\newcommand {\lk}{\operatorname{lk}}
\newcommand{\Distr}{\operatorname{Distr}}
\newcommand{\Cont}{\operatorname{Cont}}
\newcommand{\supp}{\operatorname{supp}}
\newcommand{\Int}{\operatorname{Int}}
\newcommand{\Map}{\operatorname{Map}}
\newcommand{\const}{\operatorname{const}}
\def\qed{\hbox{}\hfill\framebox{}}

\newtheorem{Thm}{Theorem}[section]
\newtheorem{Lem}[Thm]{Lemma}

\newtheorem{Cor}[Thm]{Corollary}

\newtheorem{Prop}[Thm]{Proposition}
\newtheorem{Rem}[Thm]{Remark}

\begin{document}
\title{Topologically Trivial Legendrian Knots}
\author{Y. Eliashberg\thanks{Partially supported by the NSF grants 
DMS-0707103 and DMS 0244663}\\Stanford University\\USA \and M. Fraser\\National Polytechnical Institute\\Mexico}
\date{November, 2008}

\maketitle
This paper deals with topologically trivial Legendrian knots in tight and overtwisted contact 3-manifolds.
The first parts (Sections~\ref{sec:prelim}-\ref{sec:main})  
contain a thorough exposition of the proof of the classification of topologically
trivial Legendrian knots (i.e. Legendrian knots bounding embedded
2-disks) in tight contact 3-manifolds (Theorem~\ref{thm:main}), and, in particular, in the standard contact $S^3$. 
These parts were essentially written more than 10 years ago, but only a short version
\cite{[EF]}, without the detailed proofs, was published. In that paper we also briefly  discussed   Legendrian knots in overtwisted contact 3-manifolds. The final   part of the present paper (Section \ref{sec:twist}) contains a more systematic discussion of the overtwisted case.
  In \cite{[EF]} Legendrian knots in overtwisted manifolds were divided into two classes: \emph{exceptional}, i.e. those with tight complement, and the complementary class of \emph{loose} ones. Loose knots can be {\it coarsely}, i.e. up to a global coorientation preserving
 contact diffeomorphism,  classified (see Section \ref{sec:twist} and also \cite{[D2]}) using the classification of overtwisted contact structures from \cite{[E5]}.   On the other hand, Giroux-Honda's  classification of tight contact structures on solid tori (see 
 \cite{[Gi3], [Ho1],[Ho2]})
allows us to  completely coarsely classify topologically trivial exceptional knots in $S^3$. In particular, we show the latter exist for only one overtwisted contact structure on $S^3$.  
 
Since the paper  \cite{[EF]}, several new techniques
(notably the Giroux--Honda method of convex surfaces, dividing curves and Legendrian bypasses) were developed which provide some shortcuts to the results of that paper, at least for knots in $S^3$. However, we think that our explicit geometric methods  may still have more than
just  a historic interest.  Theorem~\ref{thm:main}
 asserts that topologically trivial Legendrian knots in a tight contact manifold, which are isotopic as
framed knots are Legendrian isotopic.
A special case of this result was proved in \cite{[E1]} (for topologically trivial
 Legendrian knots with maximal possible value of the Thurston-Bennequin invariant, i.e.
$(tb,r)=(-1,0)$; see
Section~\ref{sec:invts} below). The result itself was then stated in \cite{[E2]}, where
an analogous theorem for  transversal
knots was proved.
  The proof in the present paper follows the
general scheme  of the proof
 first given in 
\cite{[F1]}.  The same classification result does not hold for non-trivial knot type, in other words the so-called {\it regular isotopy} of topologically non-trivial Legendrian knots does not imply their Legendrian isotopy. The first example of this kind was  established  by Yu.~Chekanov in \cite{[Ch2]}   using a  differential graded algebra of    a Legendrian knot, a new invariant inspired by the theory of holomorphic curves, which Chekanov defined combinatorially. 

Since the late 90's, many new invariants of Legendrian isotopy have been defined, starting with the Legendrian contact homology algebra (see \cite {[Ch2],[Ch3]} and \cite{[E9],[EGH]}) and several variations of it  (see \cite{[Sa], [EES1], [EtNS], [Ng]}.  Following some ideas from \cite{[E7]}, a different type  ``decomposition invariant'' was defined by  
Yu.~Chekanov and P. Pushkar in \cite{[ChP]} for their proof of Arnold's four-cusp conjecture. 
More recently  P. Ozsv\'ath, Z. Szab\'o and D. Thurston defined a new invariant for Legendrian knots in $S^3$ using a combinatorial version of knot Floer homology \cite{[OST]}.
  Besides the result on the unknot,   Legendrian knots are now  classified in some other topological isotopy classes of knots  in tight contact 3-manifolds,  
 Namely, J. Etnyre and K. Honda in \cite{[EtHo1]} classified Legendrian   torus knots and figure-eight knots.  F. Ding and H. Geiges in \cite{[DG]}  extended Etnyre-Honda results for other special classes of Legendrian knots and links.
 
 The problem of classification of Legendrian  (and Lagrangian)
knots
 was first explicitly formulated by V. I. Arnold in \cite{[A1]}. However,     problems
 related to Legendrian isotopy were studied earlier (see, for instance,
\cite{[A2], [A3], [A5],  [E4],  [B], [Gr1]}, and also papers \cite{[E7], [E8]}, which
were   written much earlier than they appeared). 
Legendrian knots nowadays  is a very active subject, and we will not even attempt
 here to list all relevant results. We refer the reader to Etnyre's paper \cite{[Et2]} for a survey of the status of Legendrian isotopy problem back in 2003.  The latest developments of the theory include, as we already mentioned
 above, applications of Heegaard Floer homology theory, as well as new SFT-inspired  Legendrian  isotopy invariants. 
\medskip
\tableofcontents

\section{Preliminaries}\label{sec:prelim}
Today there exist several good references  for the basic facts and notions about  contact 3-manifolds  
  (see, for instance, the book \cite{[Ge]} of H. Geiges). However, to fix the notation and terminology we review the necessary introductory information in this section. 
 
\subsection{Contact structures}\label{sec:contact-basic}
Let us recall that a {\it contact structure} on a connected $3$-manifold $M$ is a
completely non-integrable
tangent plane field $\xi$. Such a pair $(M,\xi)$ is called a
 {\it contact manifold}. The contact structure $\xi$ is said to be
{\it co-orientable}
when the bundle $\xi$ is co-orientable in $TM$.
In this case $\xi$ can be defined (as $\xi = \mbox{ker}(\alpha)$) by a global $1$-form $\alpha$ which is
called    a  {\it contact form}. In the present paper,
{\it we will henceforth assume all contact structures considered to be
co-orientable, unless otherwise specified}. The main results of the paper can be easily reformulated for the
  non-cooriented case.

Note that the orientation of $M$ defined by the volume form
$\alpha\wedge d\alpha$ depends only on the contact structure $\xi=\{\alpha=0\}$, and not on the choice of the contact form $\alpha$. When $M$'s orientation has been fixed a priori,
and happens to agree with that determined by $\xi$, $\xi$ is said to be
{\it positive};
when it disagrees, $\xi$ is said to be {\it negative}.
In the
present paper, {\it we will consider only positive contact structures};
i.e. we will always endow $M$ with the orientation determined by $\xi$.

Let  us denote by ${\Diff}_0(M,\xi)$
the connected component of the identity of the group
of {\it contactomorphisms} (i.e. group of $\xi$-preserving diffeomorphisms) of
$(M,\xi)$. This group acts transitively
 on points of any connected contact manifold (Darboux' theorem).
 Moreover, coordinates can
always be given locally so that $\xi = \{ dz=ydx\}$.
The space $\R^3$ endowed with the contact  structure $\xi = \{ dz=ydx\}$ is called
the {\it standard contact $3$-space}. 
The contact form $dz+\rho^2d\varpi$ in cylindrical coordinates defines on $\R^3$ a contact structure equivalent to the standard contact structure $\xi$.
On a closed manifold there are no deformations  of contact structures (J. Gray's
 theorem, \cite{[Gray]}):
{\it two contact structures on a closed contact manifold which are homotopic
in the class of contact structures are isotopic}.

It has proven useful (see \cite{[E2], [E3], [Gi1]}) to  divide contact
structures on 3-manifolds into two complementary classes: {\it tight} and
{\it overtwisted}. A contact structure $\xi$ is called {\it overtwisted} if
there exists an embedded 2-disk $D\subset M$ such that the boundary
$\partial D$ is tangent to $\xi$ while the disk $D$ is transversal to $\xi$
along $\partial D$. Equivalently, one can request that $D$ be everywhere tangent to $\xi$ along $\partial D$. As it is shown in \cite{[E3]} one can always find a disc in an overtwisted contact manifold such that the characteristic foliation on it (see
Section~\ref{sec:fol}) is given by Figure~\ref{fig-otwist}.

\begin{figure}[ht!]
\centerline{\epsfxsize=6cm \epsfbox{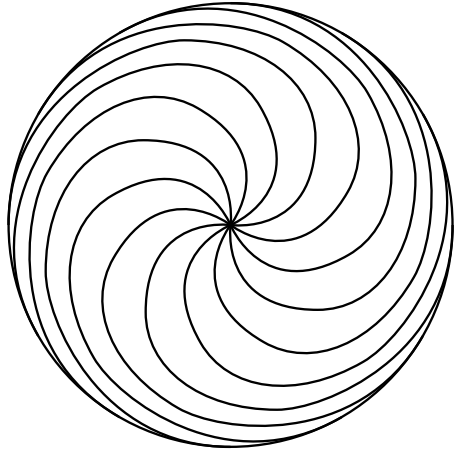}}\medskip
\caption{ An overtwisted
disk}\label{fig-otwist}\end{figure}

Non-overtwisted contact structures are called {\it tight}.
Most of this paper deals only with tight contact structures, however in
the final portion
(Section~\ref{sec:twist}) we discuss the situation for Legendrian
knots in overtwisted contact manifolds.

\subsection{Legendrian curves and characteristic foliation}\label{sec:fol}
A curve $L$ in a contact 3-manifold $(M,\xi)$ is called
{\it Legendrian} if
it is tangent to $\xi$. An oriented, closed, smooth Legendrian curve
 will be called a
{\it Legendrian knot}. In this paper we will study {\it topologically trivial }
Legendrian knots, by which we mean those Legendrian knots
 which bound embedded disks in $M$. We will sometimes use the expression
{\it Legendrian segment} to refer to a Legendrian embedding of a closed
segment of ${\bf R}$.

 According to a theorem of Rashevskii-Chow 
(see \cite{[Chow], [Ra]}), {\it  any
  embedded curve $\Gamma$ can be made Legendrian
by a $C^0$-small isotopy.} This statement    also holds  for families of curves, but not in a relative form: two isotopic Legendrian curves are not, necessarily, Legendrian isotopic. According to the Darboux-Weinstein theorem, {\it diffeomorphic Legendrian submanifolds have contactomorphic neighborhoods.}

Let us mention the two  standard ways of visualizing Legendrian curves in the
standard contact $\R^3$.  Let $L\subset \R^3$ be a Legendrian curve.
Its orthogonal projection to the $(x,y)$-coordinate plane is called the
 {\it Lagrangian projection}.\footnote{Alternatively, if instead
of standard coordinates we are using
cylindrical coordinates $(\rho,\varphi,z)$ on ${\bf R}^3$, with contact structure
defined  by the form $dz + \rho^2d\varphi$, then the orthogonal projection
to the $z=0$ plane would likewise be called Lagrangian projection.}
We will denote the projection by $p_{\rm Lag}$ and denote the image $p_{\rm Lag}(L)$
 by $L_{\rm Lag}$. The projection to the $(x,z)$-coordinate plane is
called the {\it front projection}.  We denote this projection by $p_{\rm Front}$ and call the image $L_{\rm Front}=
p_{\rm Front}(L)$ the {\it wavefront }  of $L$.

Let us observe that $L_{\rm Lag}$ is an immersed curve,
 because the contact planes are transversal
to the $z$-direction. The Legendrian curve $L$ can be reconstructed
 from $L_{\rm Lag}$ by adding
the coordinate $z=\int\limits_{L_{\rm Lag}}ydx.$ In particular, for a closed
$L$ we should have $\oint\limits_{L_{\rm Lag}}ydx=0$, i.e the algebraic area
bounded by the immersed curve $L_{\rm Lag}$ should be equal to $0$. Notice also that
 the Legendrian curve $L$ is embedded  iff   each self-intersection point
of $L_{\rm Lag}$ divides this curve into two parts of non-zero area.

On the other hand, the wavefront $L_{\rm
Front}$   may have singularities. If the
wavefront is smooth then it has to be a graph of a function $z=\alpha(x)$ defined on an
interval $I$ of the
$x$-axis. 
In the general case the front can be viewed as the graph
of a multivalued function. Generically, different branches join
pairwise in cusp-points, so that the wavefront near a singular point is diffeomorphic to
the semi-cubic parabola. In this generic case, therefore,
 the Legendrian curve $L$ can be reconstructed
from its wavefront by adding the $y$-coordinate equal to the slope of this multivalued
function. 
 Notice that the Legendrian curve $L$ is embedded iff each self-intersection point
of $L_{\rm Front}$ is transversal.


An important way in which Legendrian curves
arise is
as follows. Let $F\subset M$ be a 2-surface. If $M$ is transverse to $\xi$
then $\xi$  intersects $T(F)$ along a line field $K\subset T(F)$ which
integrates to a 1-dimensional foliation called the {\it
characteristic foliation} of $F$ and denoted $F_\xi$. The leaves of
$F_\xi$ are, of course, Legendrian curves.

Note that the foliation $F_\xi$ may
have singularities. A generic surface $F$ will be transversal to
$\xi$ except possibly at a discrete set of points where it is tangent.
 The characteristic foliation that results in this generic
case is a singular foliation with singularities  occurring
exactly at these isolated points of tangency
 (see \ref{sec:singularities} below).

The  characteristic foliation  $F_\xi$ determines a germ of contact structure
$\xi$ along $F$ (see \cite{[Gi2]}): a diffeomorphism between
 characteristic foliations  on surfaces $F$ and $\tilde F$
 extends as a contactomorphisms between the neighborhoods of
these surfaces.\footnote{
One should be more accurate in what the diffeomorphism means near singular points
\nobreak of the characteristic foliation, see \cite{[Gi2]} for the details.}

 On the other hand, if two characteristic foliations $F_\xi$ and $\tilde F_\xi$ are
{\it homeomorphic} then there exists (see \cite{[E5]})
 a  $C^0
$-perturbation $\hat F$ of the surface $\tilde F$ such that the characteristic foliations
$F_\xi$ and $\hat F_\xi$ are {\it diffeomorphic}.

\subsection{Singularities of the characteristic foliation}\label{sec:singularities}
\subsubsection{Elliptic  and hyperbolic singularities} \label{sec:ellhyp}
The index of the line field $K = \xi \cap T(F)$ at isolated singularities is well defined.
 Generically, it is equal to $+1$ or $-1$ and the
singular point is called respectively {\it elliptic} or {\it hyperbolic} in
these cases (see Figs.~\ref{fig-ellhyp}(a),(b) for elliptic, and Fig.~\ref{fig-ellhyp}(c) for hyperbolic).

\begin{figure}[hbt]
\centerline{\epsfxsize=10cm \epsfbox{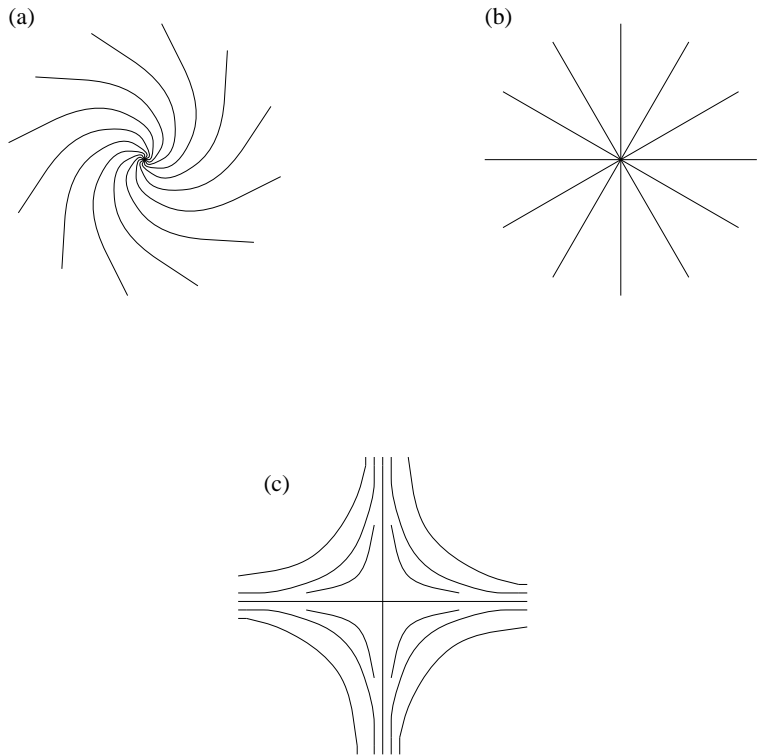}}\medskip
\caption{Elliptic singularities (top) and a
hyperbolic singularity (bottom)}
\label{fig-ellhyp}
\end{figure}

The two elliptic points shown in Figs.~\ref{fig-ellhyp}(a),(b)  are topologically
indistinguishable. Moreover, a surface can always be $C^1$-perturbed near an
elliptic point to achieve characteristic foliation of the (b)-type.
Near an elliptic point of (b)-type  one can choose Darboux cylindrical  coordinates
 $(\rho,\varphi,z)$
with the origin at
the singular point, so that the contact structure is defined by the contact form
$dz+\rho^2d\varphi$ and
 the surface $F$ is given by the equation $z=0$. We will always assume
in this paper
that  elliptic singularities have this form.

We also assume  that  characteristic
foliations we consider do not have separatrix connections between hyperbolic points.
This generic assumption  can be  achieved by a $C^\infty$-small perturbation
of the surface.

\subsubsection{The sign of a singular point} Suppose the surface $F$ is oriented.
Recall that $\xi$ is assumed to be co-oriented and $M$ oriented. So $F$ is
co-oriented in $M$. Also, $F_\xi$ is co-oriented in $F$ and so has a canonical
orientation determined by that of $F$.
 The singularities of $F_\xi$ will be called {\it positive} or {\it
negative} when the co-orientations of $F$ and $\xi$ respectively agree or
disagree at these points.
Since we are moreover assuming $\xi$ to be positive (i.e. $\alpha\wedge d\alpha
> 0$) it follows that the signs of elliptic singularities and the orientation
of $F_\xi$ are related:

\begin{Lem}
Positive elliptic
points are sources and negative elliptic points sinks of the
characteristic foliation $F_\xi$ of an oriented surface $F$ in $(M,\xi)$
(based on the assumption that $\xi$ is positive).
\end{Lem}

Notice that the sign of a hyperbolic point cannot be seen from the
$C^0$-topology of the oriented
characteristic foliation, because it is a $C^1$-, and not a $C^0$-invariant (see \cite{[Gi2]}).


\subsection{Legendrian isotopy}\label{sub:isoTrees}
\subsubsection{Legendrian isotopy vs. ambient contact isotopy}
 The following result is easily
shown using the relative version of Darboux--Gray's theorem.

\begin{Lem}\label{lem-gray}
Let $L_t , t\in[0,1]$ be a Legendrian isotopy.
Then there exists an ambient contact isotopy
 $f_t , t\in[0,1]$, in the space of contactomorphisms of $(M,\xi)$ to
itself such that
$f_0 ={\rm id}$ and $f_t(L_0)  = L_t$ for $t\in [0,1]$.
\end{Lem}

 The classification of Legendrian knots up to Legendrian isotopy is thus
equivalent to
their classification up to ambient contact
isotopy; i.e., $L \sim L' \Leftrightarrow
L' = \phi (L)$ for some
global contactomorphism $\phi$ which is contactly isotopic to the identity.

\subsubsection{Legendrian trees}

 A {\it Legendrian  graph } (in particular, a {\it tree}) in a contact $3$-manifold
$(M,\xi)$  is a  graph (tree)  $L$ embedded
in $M$ in
such a way that all its edges are Legendrian segments, non-tangent to 
 each other  at the vertices.
We call two  graphs {\it diffeomorphic} if there exists a diffeomorphism between
their neighborhoods which moves one of the trees into the other.  Thus,
 diffeomorphic graphs not only have isomorphic
structure as abstract trees but also  have the same infinitesimal
 structure at the corresponding vertices.
For instance, for a pair of corresponding   vertices $p\in L$ and $p'\in L'$
the configurations of vectors in $\xi_p$ and $\xi_{p'}$
that are tangent to the Legendrian segments
beginning at these points, should be  linearly isomorphic.

If a  characteristic foliation $F_\xi$  has no closed leaves then
separatrices  of hyperbolic points form a Legendrian graph.
In Section~\ref{manipFoln}, we will consider
certain Legendrian graphs (which will turn out to be trees)
 formed by leaves of a characteristic foliation.

\begin{Lem}\label{lm:Leg-ambient} If two Legendrian trees $L$ and $L'$ in $M$
 are diffeomorphic,
then they can be deformed
one into the other via an ambient contact isotopy.
\end{Lem}

{\it Proof.\ \ \ } Any two Legendrian segments are obviously
 Legendrian isotopic,
 and according to Lemma~\ref{lem-gray},
 they can be moved one into the other via an ambient contact isotopy.
  In fact, this remains true even if the curves
 coincide near one of the ends and the isotopy is
required to be fixed near this end.
Let us enumerate the Legendrian segments forming the trees
$$L=\mathop{\bigcup}\limits_1^N L_i,\quad L'=\mathop{\bigcup}\limits_1^N L'_i$$
 in such a way that the diffeomorphism of our hypothesis takes
${L_i}'$ to $L_i$, for all $i=1,\dots,N$, and moreover
there is a vertex $p$ and some $r\leq N$, such that $L_1\cap L_i= \{p\}
\ \forall i=2,\dots,r$ while $L_1\cap L_j=\emptyset \ \forall j>r$.
We begin by deforming $L_1'$ into $L_1$ via an ambient isotopy.
 Next we want, keeping a neighborhood $U$ of $L_1'$
fixed, to deform  $L'$ into $L$ in a neighborhood of the vertex $p$.
The contact structure  in a  small neighborhood $U\ni p_2$ can be defined in Darboux coordinates,
$\xi=\{dz=ydx\}$, so that the plane $\Pi=\{z=0\}$ can be identified with
the contact plane  $\xi_{p}$.
Let us project $L\cap U$ and $L'\cap U$ to this plane along the $z$-axis, and denote by
$\tilde L$ and $\tilde L'$ their images.
By  the assumption there  is a (germ of a) diffeomorphism $h:\Pi\to \Pi$ which sends
$\tilde L'$ into $\tilde L$. One can assume that $h$ is  fixed at the points of the
 projection of $L_1'$, preserves the area
form $dx\wedge dy$, and is isotopic to the identity in the space of diffeomorphisms
with these properties.
This $h$ can be canonically lifted to a local contactomorphism $H:U\to U$,
defined by the formula $H(x,y,z)=(h(x,y),z+f(x,y))$, where
$f$ is determined by the conditions that $h^*(ydx)=df$ and  that $f$ vanish at the origin.
The contactomorphism
 $H$ sends $L'\cap U$ into $L\cap U$, and it is contactly isotopic to the identity.
Thus  we can use $H$ to make $L$ and $L'$ coincide along $L_1$, and in the
neighborhood of the vertex $p$.
Continuing  this process inductively (at each stage of induction using in the
role of $U$ a neighborhood that includes all edges arranged in
previous levels of induction),
we construct the required isotopy between the
Legendrian  trees
$L'$ and $L$.
\qed

\subsubsection{Elliptic pivot lemma}
 As was mentioned in Section~\ref{sec:ellhyp},
 we assume in this paper that a ($C^1$-small) perturbation of
a surface
near an elliptic point, if needed, has already been
performed so that the elliptic point is
of so-called (b)-type; in which case,
 we can choose Darboux cylindrical  coordinates
 $(\rho,\varphi,z)$
with the origin at
the singular point, so that the contact structure is defined by the contact form
$dz+\rho^2d\varphi$ while
 the surface $F$ is given by the equation $z=0$.

Let us denote by $L_c$ the piecewise-smooth Legendrian curve in $F$
 which consists of two rays
$\varphi=0$ and $\varphi=c$, where $c\in (0,\pi]$. In particular, $L_\pi$ is just a
Legendrian line. Notice that  for any smooth curve $\Gamma\subset F$ there is a Legendrian
curve $\Lambda\subset \R^3$ whose Lagrangian projection
equals $\Gamma$. Indeed, given the smooth curve
$\Gamma\subset F$, we add the coordinate
 $z=\int\limits_\Gamma \rho^2d\varphi$ and consider
the Legendrian lift of $\Gamma$. Suppose that a smooth embedded curve
$\tilde L_c\subset F$  approximates $L_c$ and coincides with $L_c$ outside a small
neighborhood $D_\varepsilon=
\{\rho\leq\varepsilon\}\subset F$. Suppose also that
$\int\limits_{\tilde L_c\cap D_\varepsilon}\rho^2d\varphi=0.$
Then $\tilde L_c$ lifts to an embedded Legendrian curve $\hat L_c$ which approximates
$L_c$ and which coincides with $L_c$  outside an $\varepsilon$-neighborhood of the origin.
Thus, retaining the definition of $L_c$ above, we have:

\begin{Lem}\label{lem-pivot}  For any $\varepsilon>0$ there exists a  Legendrian isotopy $\hat L_c,\,
c\in (0,\pi]$, such that $\hat L_\pi=L_\pi$ and  for all  $c\in (0,\pi]$
the curve $\hat L_c$ coincides with $L_c$ outside  of the $\varepsilon$-neighborhood of
the origin.
\end{Lem}

\subsection{The invariants $tb$ and $r$ and Main Theorem}\label{sec:invts}
\subsubsection{The invariants $tb$ and $r$}\label{sec:tbr}
The two classical invariants of Legendrian  knots are defined
as follows.
Let $L$ be a Legendrian knot in $(M,\xi )$ which is homological to 0.
Suppose $L'$ is the
result of slightly pushing $L$ along some vector field transversal to $\xi$.
The intersection number $tb(L)$ of $L'$ with a spanning surface $S$ for $L$
is independent
of the choice of this vector field and of the surface $S$. It is  called the
{\it Thurston-Bennequin invariant}
\footnote{The invariant $tb$ is closely related
to Arnold's invariant $J^+$ and the Legendrian linking polynomial defined in
$S^1\times\R^2$; let us point out that for topologically trivial Legendrian
knots, this Legendrian linking polynomial provides no additional information
beyond $tb$. } of $L$. Equivalently,
$tb(L)$ is
the number of clockwise (positive) $2\pi$ twists of $\xi$ with
respect to the natural framing 
along $L$ induced by a spanning
surface for $L$.

Let $\beta\in H_2(M,L)$ be a
relative homology class and
$F$ a surface in the class of $\beta$.
Suppose $\tau$ is a positive
tangent vector to the oriented curve $L$. The degree $r(L|\beta)$ of $\tau$
with respect to a trivialization of the bundle $\xi |_F$ does not depend on
the choice of trivialization. Nor does it depend on the choice of a representative $F$ of
the class $\beta$. It is called the {\it rotation number} (or {\it Maslov index}) of
$L$ computed with respect to $\beta$.
If $\tilde\beta$ is another class from $H_2(M,L)$ then we have
$$r(L|\beta)-r(L|\tilde\beta)=e(\xi)[\beta-\tilde\beta].$$
Here $e(\xi)\in H^2(M)$ denotes the Euler class of the
$2$-dimensional oriented bundle $\xi$,
and the difference $\beta-\tilde\beta$ is considered as an absolute
class from $H_2(M)$. For homotopically trivial knots, one can always  choose $F$ to be
a (not necessarily embedded) disk. As it is proven in \cite{[E3]}, the Euler class
$e(\xi)$ of a tight contact structure $\xi$ vanishes on spherical classes
and thus $r(L|\beta)$ in this case is independent of $\beta$.
In this case (or if $\beta$ is clear from context), we write simply $r(L)$
instead of $r(L|\beta)$.
Notice that  $r(L)$ changes sign when the
orientation of the knot $L$ is reversed,
while $tb(L)$  is independent of the orientation of $L$.

\subsubsection{Main Theorem}
Our main result is the following:

\begin{Thm}\label{thm:main}
Let $L$ and $L'$ be two topologically trivial Legendrian knots in a tight contact
$3$-manifold. If $tb(L) = tb(L')$ and $r(L)=r(L')$ then L and L' are Legendrian
isotopic.
\end{Thm}
Theorem~\ref{thm:main} is proved in
Section~\ref{sec:mainProof}.

\subsubsection{The range of the invariants $tb$ and $r$}

The following inequality was proved by D. Bennequin (see \cite{[B]})
 for the standard contact structure on $S^3$ and
 was generalized  in \cite{[E3]} to general tight contact manifolds:

\begin{Thm}\label{thm:ineq}
Let $(M,\xi)$ be a tight contact manifold and
let $L$ be a Legendrian curve which is homological to $0$ in $M$. Then for any
homology class $\beta\in H_2(M,L)$ and surface $F\in\beta$, we have
$$ tb(L) \leq -\chi(F) -|r(L|\beta)|.$$
\end{Thm}

We remark that the numbers $tb$ and $r$ also satisfy the following congruence relation
(follows, for instance, from Lemma~\ref{lem-count} below).

\begin{Prop}\label{prop:congr}
Let $(M,\xi)$ and $L$ be as above.
Then
$$tb(L)+r(L)\equiv 1({\rm mod} 2). $$
\end{Prop}

Let  $[{\mathcal E}]$ denote the set of  isotopy classes of (conventional) knots in $M$
and let $[{\mathcal L}]=\pi_1({\mathcal L})$
denote the set of Legendrian isotopy classes of Legendrian knots.

Suppose, for simplicity, that $e(\xi)=0$.
 Then we have a map
$$\lambda :[{\mathcal L}]\rightarrow [{\mathcal E}]\times{\Z}\times{\Z},$$
where the first factor gives the topological class associated to a given
Legendrian class and the second  and the third factors
give the values of the invariants $r$ and $tb$ on this class.

The  inequality (Theorem~\ref{thm:ineq}) and the congruence
 (Proposition~\ref{prop:congr}) impose restrictions on the
image of the map $\lambda$. However, these are
 are not the
only restrictions. Additional restrictions have also been found by Y. Kanda (see \cite{[Ka]}),
D. Fuchs and S. Tabachnikov (see \cite{[FuTa]}), P. Lisca and G. Mati\'c (see \cite{[LiMa1], [LiMa2]}), L. Rudolph (see \cite{[Ru1], [Ru2]}) and for the analogous map in the case of Legendrian links K. Mohnke \cite{[Mo]}.

 Our main theorem~\ref{thm:main} states that the map $\lambda$ is injective when restricted
to
$\lambda^{-1}([0]\times{\Z}\times{\Z})$, where $[0]$ denotes the
topological class of the unknot. The map is not however injective in
general (see \cite {[Ch2], [EGH]}).

Let $\mathcal D\subset \Z\times \Z$  be the range of the invariants
$(r,tb)$ on the space of topologically trivial Legendrian knots.
Even though  Theorem~\ref{thm:ineq} and Proposition~\ref{prop:congr}
preclude surjectivity of $\lambda_0=\lambda|_{[{\mathcal L}]_0}$, where
$[{\mathcal L}]_0=\lambda^{-1}([0]\times{\Z}\times{\Z})$, the domain ${\mathcal D}$ of
$\lambda_0$ is as large as possible (see Fig. \ref{fig-catalog}). Indeed:

\begin{Lem}
${\mathcal D}=\{(m,-|m|-2k-1)\;|\; k\geq 0\}$

\end{Lem}

{\it Proof.\ \ }  The inequality of Theorem~\ref{thm:ineq} and the congruence
of Proposition~\ref{prop:congr} show that $${\mathcal
D}\subset\{(m,-|m|-2k-1)\;|\; k\geq 0\}.$$ On the other hand,
 for any pair
$(m,n)\in{\mathcal D}$ the catalog given in Fig. \ref{fig-catalog}
provides a wavefront of  a Legendrian knot $L$ in $\R^{3}$ with
$tb(L)=n,\, r(L)=m$.
 Since any contact manifold contains the standard contact $\R^3$
(by Darboux's theorem), all the
examples of the catalog can be constructed in a general $(M,\xi)$.
\qed

\subsubsection{Catalog of Legendrian unknots}\label{sec:catalog}

\begin{figure}[ph!]
\vspace{2cm}
\centerline{\epsfxsize=10cm \epsfbox{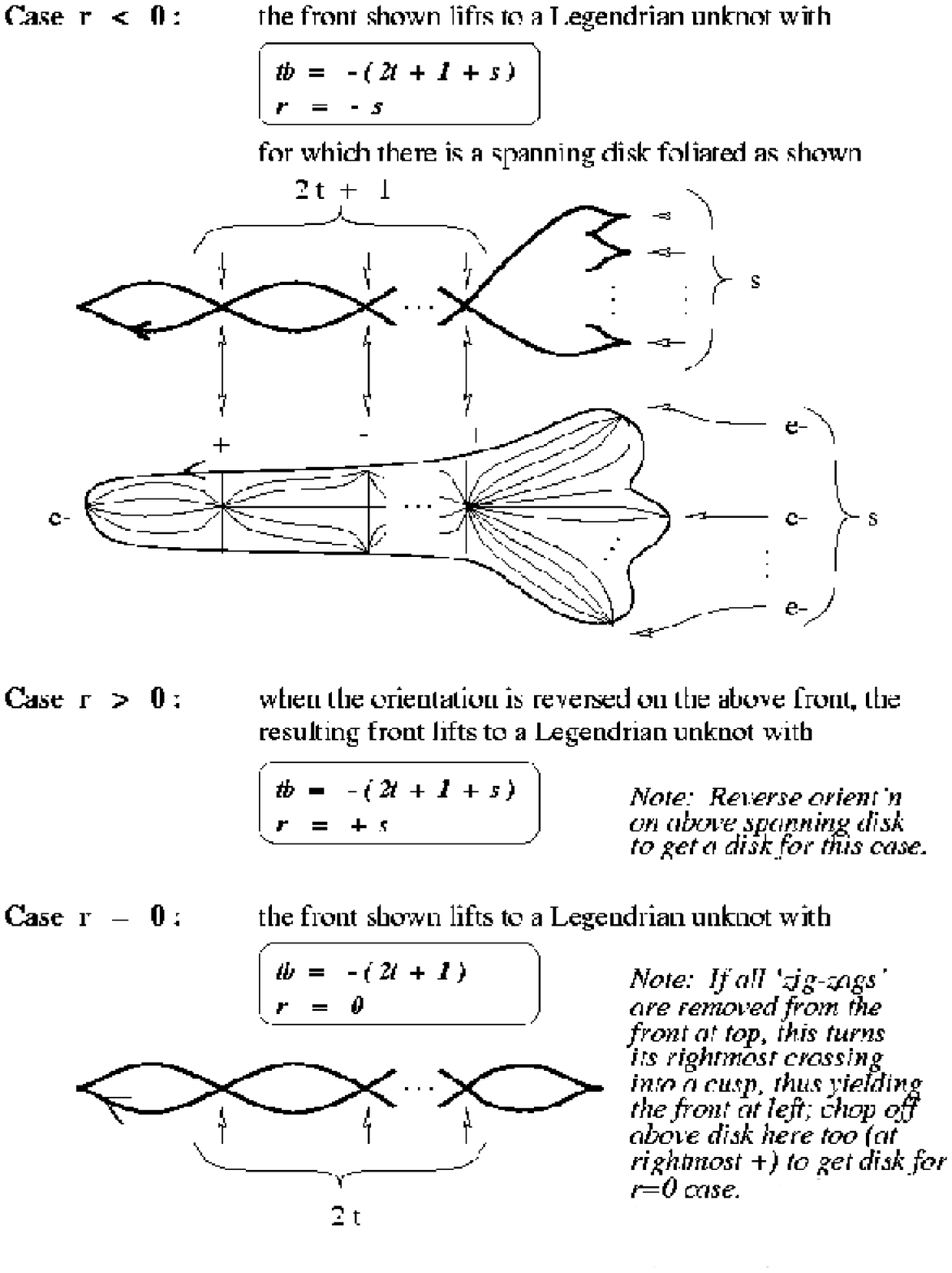}}\medskip
\caption{Catalog of Wavefronts}\medskip
\label{fig-catalog}
\end{figure}

Figure ~\ref{fig-catalog} provides a list of
Legendrian wavefronts each lifting
to a Legendrian unknot in standard $\R^3$ with specified values
of $tb,r$.

In general, the values of $tb,r$ can be read from a
 wavefront projection as
follows:

Given a Legendrian knot $L$ in $\R^3$ (or in $S^1\times\R^2=ST^*(\R^2)$)
let $\Omega$ be the
set of all points of self-intersection of the front
$L_{\rm Front}=p_{\rm Front}(L)$, and let $K$ be the set of all cusps of
the front. Note that
for the standard $(\R^3, dz-ydx)$, the projection $p_{\rm Front}$
is projection to the $(x,z)$-plane.
Now, for each
$p\in\Omega$,  define\footnote{This is equivalent to defining
$or (p) $ as the orientation determined by the pair of emanating rays from $p$
written in the
order (ray with greater slope, ray with lower slope).}   $or (p)$
to be
$+1$ or $-1$ depending on whether the two rays of $L_{\rm Front}$
emanating {\it from} $p$ lie on opposite sides of a vertical line through
$p$ or on
the same side.
These conventions are indicated in Figure~\ref{fig-frntConv}.
Also, for each $p\in K$, define $\kappa (p)$ to be
$+1$ or $-1$ depending on whether the ray emanating from $p$ lies above
or below the ray entering $p$ (i.e. has
$z$-value higher or lower than the other ray for given $x$-value).
Thus $\kappa$ indicates whether the original
curve $L$ was rising ($\kappa (p) > 0$) or falling ($\kappa (p) < 0$) as it
passed through the preimage of $p$.

Then,
\begin{eqnarray*}
  tb(L)&=&-{\sum_{p\in\Omega}}\ or(p) -{\frac{1}{2}}{\sum_{p\in K}}\ |\kappa (p)|         \\
 r(L)&=& {\frac{1}{2}}{\sum_{p\in K}}\ \kappa (p)
 \end{eqnarray*}

\begin{figure}[h!]
\centerline{\epsfxsize=6cm \epsfbox{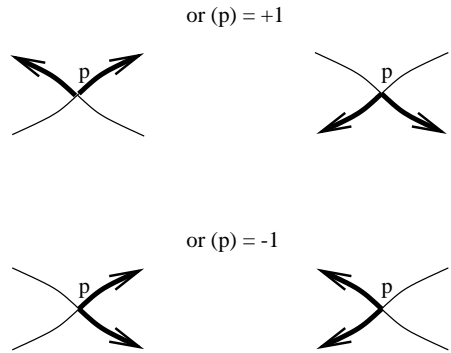}}\medskip
\caption{Conventions for Crossings of Fronts}
\label{fig-frntConv}
\end{figure}

 In the Catalog of Figure~\ref{fig-catalog}, each $(tb,r)$ case also includes a
description or a sketch of a singular foliation on a disk $D$.
It is easily checked that the Legendrian lift for each front has a spanning
disk with foliation as shown (see Section~\ref{sec:frntSpDisk} for details).


\section{Manipulation of characteristic foliation}\label{manipFoln}

The goal of this section is  Lemma \ref{lem-ellForm} below which, given a topologically
trivial Legendrian knot $L$, establishes existence of a spanning disc whose characteristic foliation
has a special form, which we call {\it elliptic}. We will achieve this
by taking an arbitrary spanning disc for $L$ and then altering its characteristic foliation via gradual
deformation of the surface.

\subsection{Birth, death and conversion}
The basic tools which will allow us to effect the desired alterations
of  characteristic
foliation  are the controlled birth and death of elliptic-hyperbolic
singularity pairs. The controlled death of such a pair, i.e. its killing, or {\it elimination},
is the subject of Lemma~\ref{lem-elim}. It was first proved by  E. Giroux in \cite{[Gi2]}, and in
a form
improved  by D.B. Fuchs, was presented in \cite{[E2]}. The creation of an elliptic-hyperbolic pair of
singularities is more local and is relatively
 straightforward; in its most
basic form (Lemma~\ref{lem-birth}) it is often stated without proof.
In the present paper, another
form of pair-creation (called {\it elliptic-hyperbolic conversion}, given by
Lemma~\ref{lem-convert}) will also be used. All three types
 (elimination, conversion, creation )
rely on the same basic idea: twisting a strip of surface along a Legendrian
curve in a manner dictated by the twisting of $\xi$ along the curve.

\begin{figure}[h!]
\centerline{\epsfxsize=12cm \epsfbox{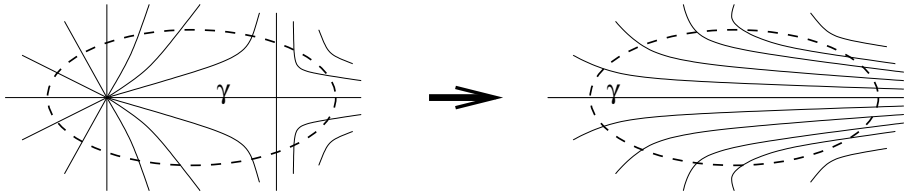}}\medskip
\caption{ Elliptic-hyperbolic elimination}\label{fig-GElim}
\end{figure}

\begin{Lem}\ {\rm (Elimination)}
Let $S$ be an embedded surface in $(M,\xi)$ such that $S_\xi$ has exactly
2 singularities, $p$ and $q$, which are respectively elliptic and
hyperbolic. Suppose that they are both of same sign and are connected by
$\gamma$, one of $q$'s
separatrices.
Then given any arbitrarily small neighborhood of $\gamma$,
there exists a $C^0$-small isotopy of $S$ supported in that neighborhood and
fixing $\gamma$ which
results in a new surface having no singularities of its characteristic foliation (see
Figure~{\rm \ref{fig-GElim}}).
\label{lem-elim}
\end{Lem}

\begin{Lem}\ {\rm (Elliptic-hyperbolic conversion)}
Let $S$ be an embedded surface in $(M,\xi)$ having just one singularity of its
characteristic foliation called $p$. Suppose $p$ is elliptic (hyperbolic).
Let $\gamma$ and
$\tau$ be two leaves of the characteristic foliation which pass through $p$,
intersecting transversally and at this point only.  Moreover,
when $p$ is hyperbolic, suppose the name $\gamma$ is assigned to the
stable (or unstable) separatrix of $p$ for the case $p$ negative (or positive).
Then
given any arbitrarily small neighborhood of $p$, there exists a $C^0$-small isotopy
of $S$ supported in that  neighborhood, fixing both
$\gamma$ and $\tau$, which results in a
new surface having a hyperbolic (elliptic) singularity at the intersection of $\gamma$
and $\tau$ as well as two additional elliptic (hyperbolic) singularities, one on each side
of $p$ as shown in Figure~{\rm \ref{fig-convert}}, all three singularities being of the same sign
as $p$ and there being no further singularities.

\begin{figure}[h!]\centerline{\epsfxsize=10cm\epsfbox{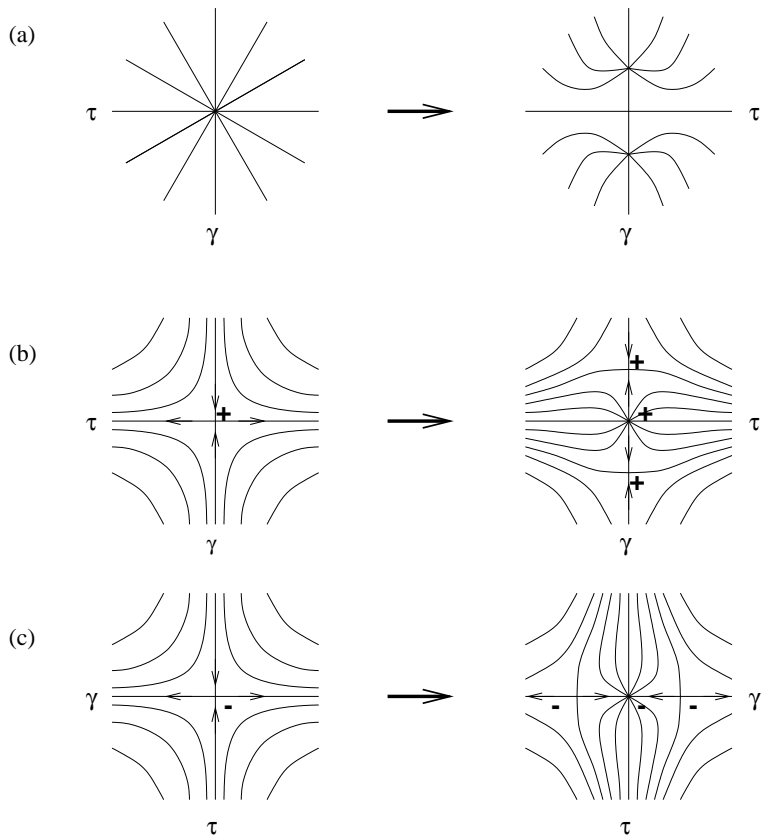}}\medskip
\caption{The cases of elliptic-hyperbolic conversion}
\label{fig-convert}
\end{figure}

\label{lem-convert}
\end{Lem}

Note that cases (b),(c) in Figure~\ref{fig-convert} can be summarized by the
observation that the newly created pair of elliptic sink (source) singularities will of
necessity be created on whichever separatrix flowed into (out of) $p$; we
have used the name $\gamma$ for it. For
case (a) we have the freedom to choose which of the leaves
will be
granted new singularities and we label it $\gamma$.
For all three cases, the process can be viewed, on the one hand, as
the {\it conversion} of the singularity-type  of $p$ together with the {\it creation} of two elliptic (or else two hyperbolic singularities) on either side
 or, on the other hand,
 as  controlled birth
of an elliptic-hyperbolic pair during which one of the two newly created singularities
`displaces'
$p$.

\begin{Lem}\ {\rm (Basic form of pair-creation)}\label{lem-birth}
Let $S$ be an embedded surface in $(M,\xi)$ having no singularities in its
characteristic foliation $S_\xi$. Choose a leaf $\gamma$ of $S_\xi$.
Then given any arbitrarily small neighborhood through which $\gamma$ passes,
there exists a $C^0$-small isotopy of $S$ supported in that neighborhood and
fixing
$\gamma$ which results in a new surface having exactly two singularities (of same sign)
which both lie on $\gamma$, one being elliptic, the other hyperbolic as shown in
Figure~\ref{fig-birth}.
\end{Lem}

\begin{figure}[h!]
\centerline{\epsfxsize=12cm \epsfbox{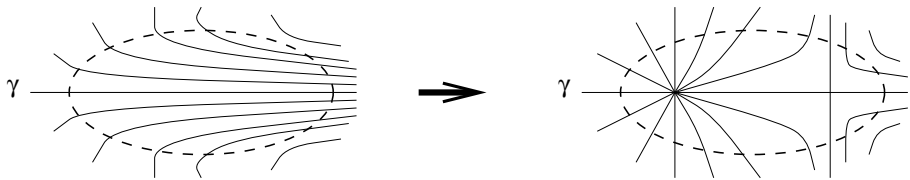}}\medskip
\caption{ Elliptic-hyperbolic creation}\label{fig-birth}
\end{figure}

\subsubsection{Singularity curves}\label{singCurves}
Given an oriented embedded surface $F$, suppose we have
the non-generic situation where there is a Legendrian curve $L \subset
F$, consisting entirely of singularities of $F_\xi$. We then say $L$ is
a {\it singularity curve} on $F$.
A model for such a situation is the $x$-axis (i.e. the line $\{y=z=0\}$)
on the surface $\{z=0\}$ in the standard $({\bf R}^3,dz-ydx)$.
This is illustrated in Figure~\ref{fig-singCurve}.

\begin{figure}[ht!]
\centerline{\epsfxsize=10cm \epsfbox{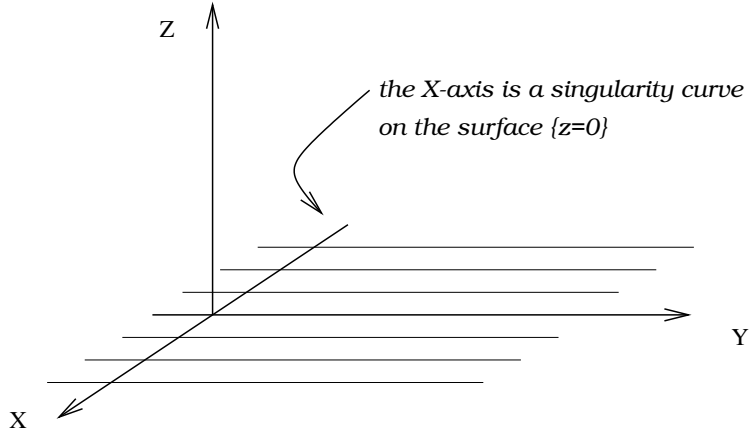}}\medskip
\caption{ Model for singularity curve in standard $({\bf R}^3, dz-ydx)$}
\label{fig-singCurve}
\end{figure}

Note that on a given oriented embedded surface $F$ all the singularities in a singularity curve are of the
same sign, and we can thus speak of positive or negative singularity curves on $F$.

\subsubsection{Proof of manipulation lemmas}\label{sec:manipProof}

{\it Proof of Lemmas~\ref{lem-elim},~\ref{lem-convert},~\ref{lem-birth}.\ \ }
First, consider a model Legendrian curve and its neighborhood.
In the standard contact $({\bf R}^3, dz-ydx)$, let $L$ be the $x$-axis
and let $N$ be a small tubular neighborhood thereof.

Let $Z$ be a cylinder of radius 1
with $L$ as its core, parameterized by $\theta, x$ where $(r, \theta)$ are polar coordinates in
the $(y,z)$-plane. Note that the
characteristic foliation of the surface $Z$ has singularity curves along the
top and along the bottom, and has no other singularities.
Specifically, $\{y = 0, z= \pm 1\}$ are singularity curves and
the characteristic foliation flows
from the positive ($z = +1$) to the negative ($z = -1$)
singularity curve, crossing $\{z=0\}$ at an angle of
${\frac{\partial\theta}{\partial x}}= 1$.

Note that along any Legendrian curve we have a dilating (in normal plane)
contact flow and the tubular neighborhood this defines is standard, i.e.
contactomorphic to such a tubular neighborhood constructed around any
other Legendrian curve, in particular to the neighborhood $W = \{z^2 + y^2\leq r\}$
around $L$, with $W = \partial Z$. 
To prove each Lemma, let $N$ be a standard neighborhood of $\gamma$, i.e.
one contactomorphic to $W$
(with $\gamma$ corresponding to $L$).

Now consider model surfaces in $W$, each of the form $A_f$ for a smooth function $f: {\bf R} \to {\bf R}$, where $A_f$ is the smooth
surface swept out along $L$ by a line normal to $L$ whose angle
with the horizontal at each point is $f(x)$. Specifically, $A_f = \{ (x, y, z):
y \sin f(x)=z \cos f(x) \} = \{ (x, r, \theta): \theta = f(x) \mod 2\pi$. These are so-called {\it staircase surfaces}.

For positive real $\epsilon$, let $g(x) = \epsilon \cos x$ 
and $h(x)= -\epsilon$.
It is easily verified that
the characteristic foliation on $A_h$ within $W$ is
singularity-free (as given for Lemma~\ref{lem-birth}), and also that
it meets $Z$ in a transverse curve.
One may also verify that for $\epsilon < 1$, the characteristic foliation on $A_g$ for
$-\pi<x<\pi$ is as given for
Lemma~\ref{lem-elim};
and finally, that $A_g$ for $-\pi < x <-{\frac{\pi}{4}}$, has characteristic foliation like that
on the left of part (a) in
Lemma~\ref{lem-convert}, while $A_g$ for ${\frac{\pi}{4}} < x  < \pi$ has
foliation as on the left of parts (b),(c).
\begin{figure}
\centerline{\epsfxsize=6cm \epsfbox{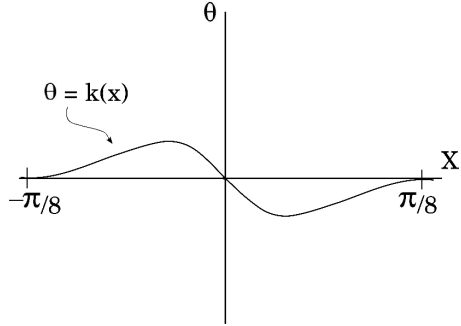}}\medskip
\caption{ The function $k(x)$ used to perturb $g(x)$ in proof of
Lemma~\ref{lem-convert} }
\label{fig-wiggle}
\end{figure}
In making this verification of characteristic foliation,
one observes
 that the hyperbolic vs. elliptic nature of the singularity
 is a result
of $g'$ being respectively negative or positive\footnote{Cases will correspond to rate of rotation of surface compared to $\xi$ having resp. different or same sign along two Legendrian axes through singularity.}
where $g(x)=0$, since $A_g$ twists in the same direction
as $\xi$ along lines $\{z = 0, x = \pm \frac{\pi}{2}\}$ albeit more slowly (due to $\epsilon < 1$).
Note that the curves $C=A_g\cap Z$
stay within an $\epsilon$-neighborhood of the horizontal curves
$H=\{z=0\}\cap Z$ and that
the slope $\frac{\partial\theta}{\partial x}$
of the curves $C$ (in the cylinder $Z$) is
$-\epsilon \sin x$, so the absolute slope is $\leq \epsilon$.
Meanwhile, the leaves of $Z_\xi$ have slope $\frac{\partial\theta}{\partial x}=1$ along $H$ and even higher slope elsewhere.
Thus, in particular, $C$ is transverse to $\xi$. In general, 
any staircase surface $A_f$, such that $|f'| < 1$, will meet $Z$ in a transversal curve.

We now define two further functions that will be used to produce the model surfaces for the right hand sides of parts (a),(b),(c) in Lemma~\ref{lem-convert}. Assume $\epsilon < \frac{1}{4}$.
Let $k(x)$ be a function as shown in Figure~\ref{fig-wiggle},
with the property that $|k'| \leq 2\epsilon$, that
 $k'(0)=-2\epsilon$, and that the absolute value of the function is bounded by $\epsilon$. 
We use $k$ to
perturb $g$ near $\pm {\frac{\pi}{2}}$. Specifically, let
$g_-(x)=g(x)+k(x+{\frac{\pi}{2}})$ in the region  $x<-{\frac{\pi}{4}}$
and let $g_+(x)=g(x)-k(x-{\frac{\pi}{2}})$ in the region $x>{\frac{\pi}{4}}$.
Then $g_-'(-{\frac{\pi}{2}})=-\epsilon$, while $g$ had slope $+\epsilon$ there.
And, $g_+'({\frac{\pi}{2}})=+\epsilon$, while $g$ had slope $-\epsilon$ there.

Moreover, $|g_-'| \leq 3\epsilon < 1$ and likewise $|g_+'| < 1$
so the surfaces
$A_{g_-}$ and $A_{g_+}$ (for the relevant $x$-regions) meet Z in transversal curves and
it is easily checked
that they exhibit the characteristic foliations respectively
for the right hand sides
of part (a) and of parts (b),(c) in Lemma~\ref{lem-convert} (with $\gamma$ as the $x$-axis $L$).

We now have model surfaces for the left and right hand sides of all
parts of the three Lemmas. Each of these meets $Z$ along curves transverse to $\xi$
and stays within a small neighborhood of $H$. For each Lemma statement,
let $B$ be the portion of $Z$ between
the two curves, and add it to the model for the left side, smoothing along
the two curves where the joining is performed. This can be done
without introducing singularities in the characteristic foliations
 since the curves are transverse to $\xi$.
In each case the new surface is a model for the right side (i.e. exhibits its foliation), contains $\gamma$,
is $C^0$-close to the old surface and coincides with it away from $L$. We thus have the required $C^0$-small deformations.

\qed

\subsubsection{Manipulation lemma for singularity curves}

With a modification of the above proof, one can also prove Lemma~\ref{lem-singCurve}
below. One uses new
model surfaces of the form $A_f$ (once again with $L$ in the role of $\gamma$)
where $f$ is defined piecewise as follows:
for values of $x$ on which we want a singularity curve use $f(x) = 0$, where we want a
half elliptic point (at $x = a$) use a function that is locally of the form $\epsilon (x-a)^3$, and
where we want a half hyperbolic point (at $x = b$) one of the form
$-\epsilon (x-b)^3$. Assume that the factor $\epsilon$ is
chosen sufficiently small (as in previous proof) and the
patching together of $f$ is done reasonably, so that the resulting $f$
has small absolute value and absolute slope and we can thereby
ensure $A_f$ meets $Z$ along
a transverse curve. The same argument as above then applies to
construct the desired surface deformation.

\begin{Lem}\ \label{lem-singCurve}
Let $S$ be an embedded surface in $(M,\xi)$ and let $\gamma\in S$ be an
embedded open interval.
Suppose the characteristic foliation near $\gamma$ is as shown on the
left (resp. right) hand side of Figure~\ref{fig-lemSingCurve}, then there is
a $C^0$-small isotopy of $S$ supported in an arbitrarily small neighborhood
of $\gamma$ that fixes $\gamma$, $\rho$,
and (where applicable) $\tau$, while resulting in
a new surface with characteristic foliation as shown on the right (resp. left)
side of this Figure.
\end{Lem}

\begin{figure}
\centerline{\epsfxsize=10cm\epsfbox{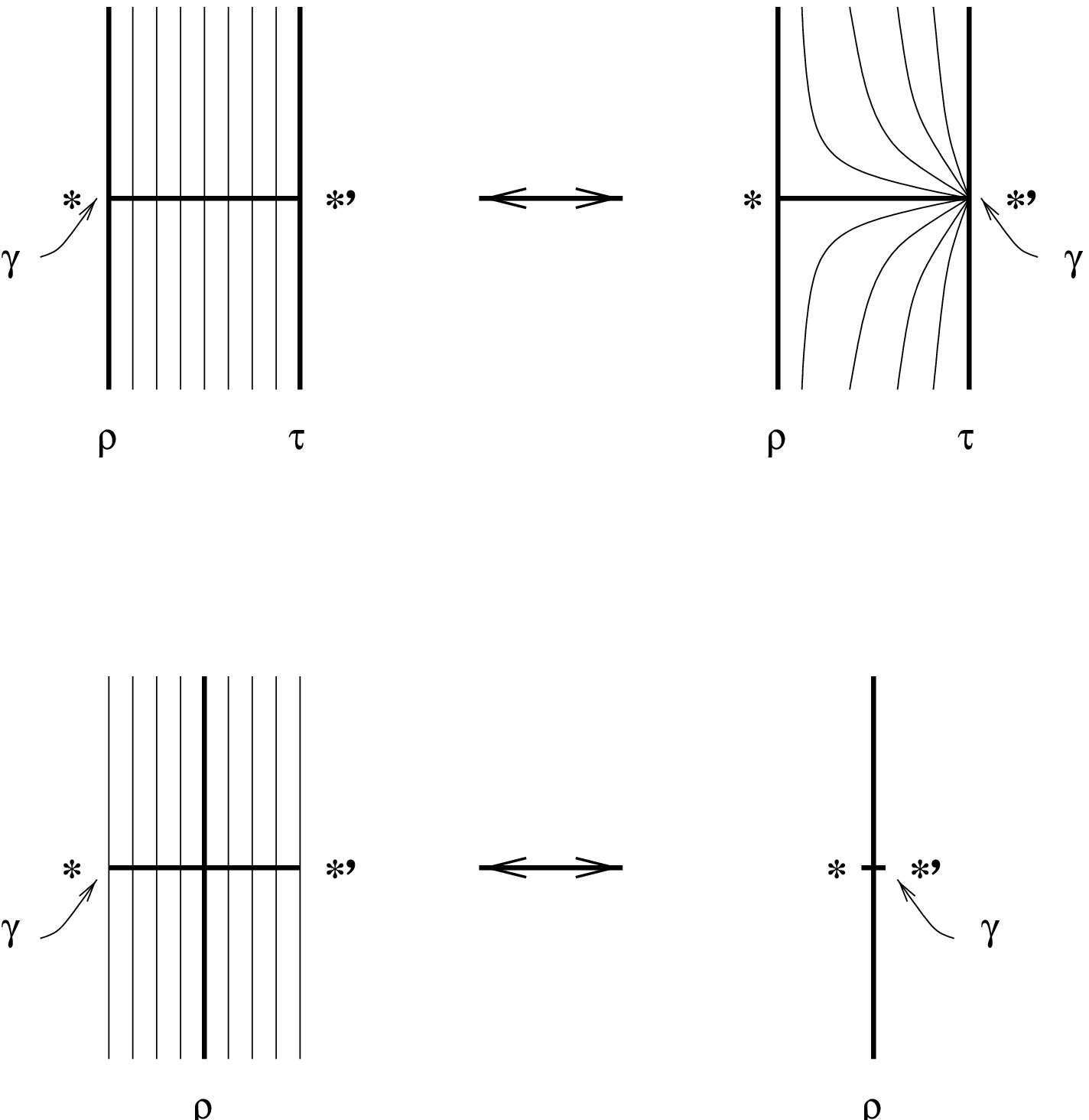}}
\medskip
\caption{Manipulations of singularity curves}
\centerline{(each *, *' means a fixed choice of elliptic or hyperbolic)}
\label{fig-lemSingCurve}
\end{figure}

\subsection{Characteristic foliation on the disk}
\subsubsection{Program for standardizing the foliation}
Given an (oriented) spanning surface $S$ of a Legendrian knot $L$,
the characteristic
foliation $S_\xi$ will be said to be in {\it  normal absorbing form} (abbrev. NAF)
along $L$
if the singularities on $L$ alternate in sign and the positive ones are
hyperbolic, the negative
elliptic.

Once we have this form of boundary foliation, we know all
flow between interior and boundary is directed towards the boundary.
This makes it easier to control the interior foliation.

In particular, by obtaining normal absorbing form on the boundary,
we are then able to eliminate all but positive elliptic and negative hyperbolic
singularities on the interior; this process is described in the proof of
Lemma~\ref{lem-stdize}.
Interior characteristic foliation that has been
reduced to this simple form, i.e. for which all singularities are
either positive elliptic or negative hyperbolic,
will be said to be {\it reduced}. Thus we will speak
of a disk whose foliation is {\it reduced with normal absorbing form on the
boundary}.
In this state,
we have control over the number and placement of the interior singularities as
we shall see in sections~\ref{sec-count} and~\ref{sec-skel}.
Then, once this control is achieved we can retain it while altering
the types of some boundary and interior singularities so as to make the over-all
foliation more suited to the present proof. This final form of foliation
will be called {\it elliptic form}; it is
the subject of the section~\ref{ellForm}.

To summarize, the sequence of characteristic foliation types we will pass
through is:
\begin{enumerate}
\item NAF on boundary:\hfill\hbox{just h+ and e- on boundary.}
\item Reduced with NAF on boundary:
\hfill\hbox{just h+ and e- on boundary,}\break
\hbox{}\hfill\hbox{and, just e+ and h- on interior.}\break
\item Elliptic form:\hfill\hbox{mostly
h+ and h- on boundary,}
\footnote{For precise definition,
see section~\ref{ellForm}.}\break
\hbox{}\hfill\hbox{just e+ and e- on interior.}\break
\end{enumerate}

\subsubsection{First steps of standardization}

The following Lemma deals with achieving steps 1. and 2. of the
above standardization program.

\begin{Lem} In tight $(M,\xi)$, let $D$ be an oriented, embedded disk with
Legendrian boundary $L$.
Then there exists a
$C^0$-small deformation of the disk which produces
a new disk whose characteristic foliation is reduced with
normal absorbing boundary.
\label{lem-stdize}
\end{Lem}

{\it Proof.} Given $D$ and $L$ as in the hypotheses, we will
first arrange that signs alternate along $L$. By definition of
$tb(L)$, this invariant counts
the net $2\pi$-twisting along $L$ of $\xi$ with
respect to the framing induced by $D$. This is the same as the
net $2\pi$-twisting along $L$ of a normal to $\xi$ with respect to a
normal to $D$. Wherever these two normals are aligned along
$L$ there will be a singularity of $D_\xi$. Let us
assume the normals to $\xi$ and $D$ that we use are those induced
by the respective co-orientations of $\xi$ and $D$, then this
singularity will be positive (negative) when the normals coincide with
(are negatives of) each other.
Note that for each $2\pi$-twist of $\xi$, there will thus be
at least 2 singularities of $D_\xi$ (one of each sign).
Moreover, by twisting $D$
along $L$,
one can clearly attain the minimal situation
where there are exactly $2|tb(L)|$ singularities along $L$
with alternating signs. Note that this deformation can be taken to be
$C^0$-small. We now suppose it has been performed, and label the resulting
disk once again $D$.

\begin{figure}
\centerline{\epsfxsize=10cm\epsfbox{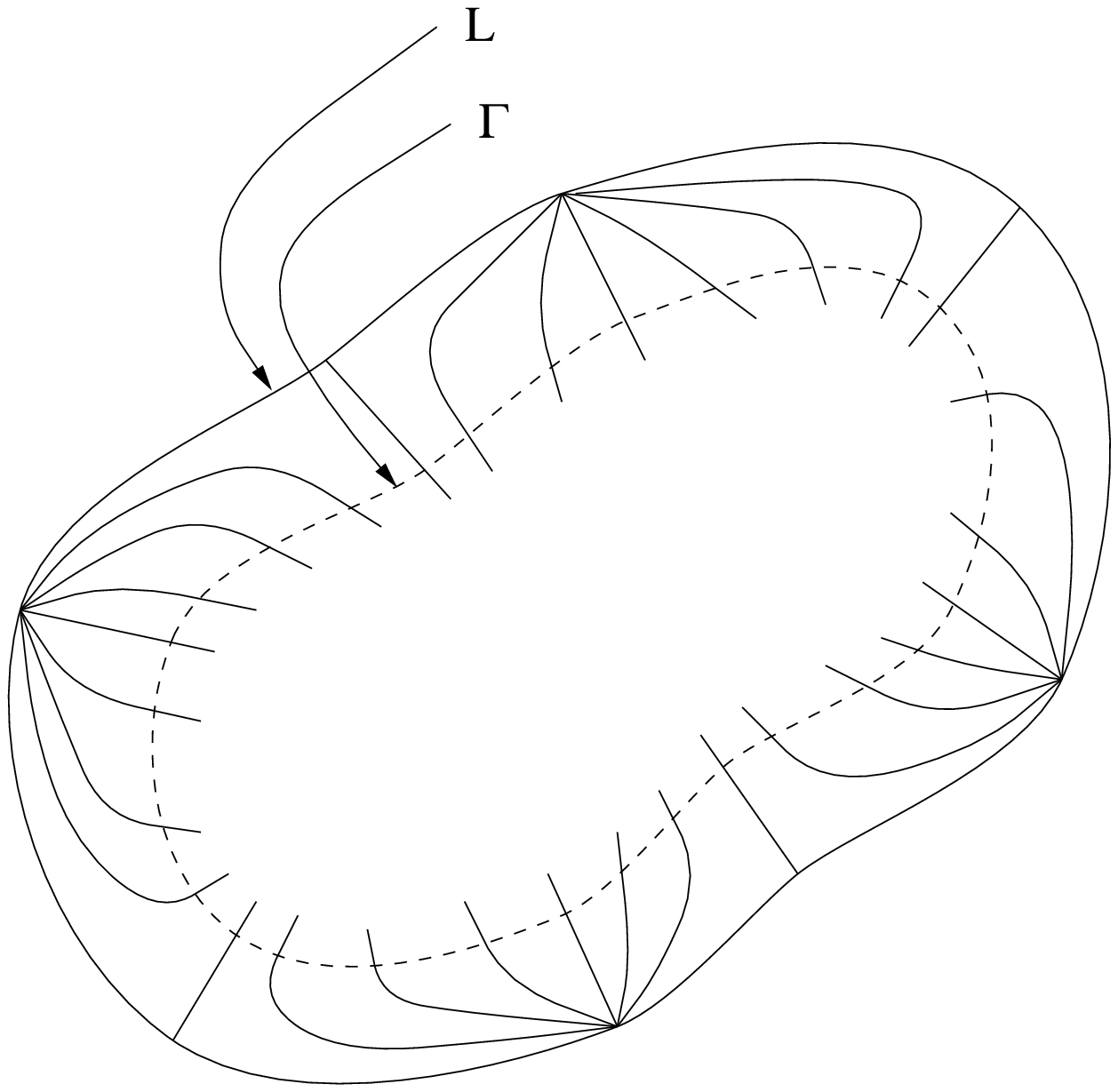}}
\caption{View of $D$ near $\partial D=L$;}
\centerline{Note normal absorbing form along $L$}
\centerline{and resulting $\Gamma$ just inside $L$.}
\label{fig-gamma}
\end{figure}

Using Lemma~\ref{lem-convert} allows us to keep $L$ fixed and
convert any positive elliptic singularities of $L$ into positive
hyperbolic singularities. Likewise, we can convert any negative
hyperbolic singularities into positive hyperbolic singularities.
If we let $L$ correspond to $\tau$ of Lemma~\ref{lem-convert} in
each case, then no new singularities will be introduced along $L$
and we will eventually, after successive conversions, reach normal
absorbing form. Note that the deformations of the disk
which accomplish these conversions can be made $C^0$-small.
Once again, label the resulting disk $D$.

Now that NAF has been achieved on $D$, there is an
obvious transversal unknot
$\Gamma$ lying on $D$ just inside $L$ (see Figure~\ref{fig-gamma}).
Because $D_\xi$ was in normal
{\it absorbing} form along $L$, the characteristic foliation flows
outward across $\Gamma$ and onto
$L$. This flow exiting across a transversal unknot
is the situation assumed in Prop. 2.3.1 and Section 4.4
of  the paper \cite{[E2]} where an account is given of how to
eliminate all interior
singularities  except for the positive elliptic and negative
hyperbolic ones. The basic idea is as follows:\break
- First, destroy all hyperbolic-hyperbolic connections\footnote{Such
connections are highly unstable; a $C^1$-small pushing upwards of the surface
at a point along such a shared separatrix will slightly shift the
characteristic flow
so that the two separatrices no longer connect to each other.}.\hfill\break
- Then, suppose there exists a negative elliptic point $p$ on the interior.
All of the arcs flowing into $p$ must originate at interior singularities,
since the flow exits across $\Gamma$ and there are no limit cycles in $D_\xi$
(by tightness). Consider the
{\it basin} of $p$, i.e. the closure of the set of all points connected to
$p$ by smooth arcs. It will be an immersed disk, with self intersection
only at (possibly) some subsegments of the boundary, and an embedding on
the interior. Let $B$ denote the boundary before immersion.
No two hyperbolic points
will be adjacent as we travel along $B$,
since all h-h connections were destroyed. No two
elliptic points can be adjacent either, since they must all be positive
(sources flowing to $p$).
Thus hyperbolic and (positive) elliptic
singularities will alternate along $B$. Moreover, there must be
at least 2 singularities in $B$.
If all $B$'s hyperbolic points were positive we
could kill them all, together with the positive elliptic points which must
separate them, thus obtaining a limit cycle. So in fact, there must exist
a negative hyperbolic point along $B$ and we can use it to kill $p$.
Thus we may assume there do not exist any negative elliptic points in the
region enclosed
by $\Gamma$.\hfill\break
- Finally, suppose there exists a positive hyperbolic point $p$ on the
interior. Look at its stable separatrices. These cannot come from a limit
cycle, a point outside $\Gamma$, or a hyperbolic point,
so they must come from an interior positive elliptic point. Kill the pair.
Thus we may assume there do not exist any positive hyperbolic points
in the region enclosed by $\Gamma$. The characteristic foliation
in this region is therefore now
in reduced form.\hfill\break
- Besides the breaking of hyperbolic connections,
 all deformations were
of the elimination type given in Lemma~\ref{lem-elim}, and so we may assume the
overall deformation required is $C^0$-small.
\qed

\subsubsection{Counting singularities}\label{sec-count}
The following count of interior singularities holds as long as we have normal absorbing
form on the boundary, whether or not the interior has been reduced.

\begin{Lem}\label{lem-count}
Let F be an embedded disk spanning the Legendrian knot L
and having normal absorbing form on L. Let $e_\pm$ and $h_\pm$ be the number of
$\pm$-ve interior elliptic and hyperbolic singularities respectively.
Then,

$$e_\pm - h_\pm = {\frac{1}{2}}(1 \mp tb(L) \pm r(L)).$$
\end{Lem}

{\it Proof.\ \ } A similar fact is mentioned for transversal knots in \cite{[E2]}.
Either version can be directly derived from the more general calculation in \cite{[HE]}.
One may also easily derive the Legendrian  version from the transversal version.
Indeed, given a Legendrian knot $L=\partial F$ with NAF, and considering a small
tubular neighborhood of $L$, one may assume standard coordinates and write down an
explicit isotopy taking the usual (see \cite{[E2]}) positive transversalization $T_+(L)$ of
$L$ (whose intersection number with $F$ is $tb(L)$) to the positive
transversal knot which lies just inside any NAF boundary on the spanning surface (this
is $\Gamma$ of Figure~\ref{fig-gamma}). In the present setting let us also refer
to this curve as $\Gamma$. Recall
  from \cite{[E2]} that ${\it l}(T_+(L))= tb(L)-r(L)$ and also the transversal version of the
counting formula we are trying to establish, namely $e'_\pm-h'_\pm= {\frac{1}{2}}(1 \mp
{\it l}(T_+(L))$, where $e'_\pm, h'_\pm$ refer to those singularities inside
$\Gamma$ . Since the small collar of surface separating
$\Gamma$ from
$L$ in Figure~\ref{fig-gamma} contains no interior singularities of $F_\xi$, the number
of interior singularities is the same for $\Gamma$ and $L$, i.e.
$e_\pm=e'_\pm, h_\pm=h'_\pm$. Thus we obtain the desired formula.
\qed

 Having the interior in
reduced form as well would simply mean that in the equation of Lemma~\ref{lem-count},
$h+=e-=0$. In the spirit of this equation and also the elimination lemma
(Lemma~\ref{lem-elim}), pairs
of opposite-type but same-sign singularities
can be viewed as ``non-essential''; since, in the
equation, the contributions
made by members of such a pair cancel each other, and by
the elimination lemma (Lemma~\ref{lem-elim})
the members themselves can sometimes be ``canceled''.
Therefore, particularly
useful forms for the interior characteristic foliation on a surface are ones
where no such non-essential pairs exist, i.e. for each sign
only one type occurs, in particular: {\it reduced form} where we have only $e+,h-$
or {\it elliptic form} (to be defined in Section~\ref{ellForm})
where we have only $e+,e-$.

\subsubsection{Legendrian tree of a reduced form disk}
 We will consider {\it Legendrian trees} (see Section~\ref{sub:isoTrees})
embedded into a spanning surface.
 We assume that vertices of such trees are located at  singularities
of the  characteristic foliation, and  each edge  is either singularity-free, or
 contains exactly
one hyperbolic singularity. Usually the trees we consider will have edges
all of the same type, either all singularity-free or all hyperbolic-containing.

When the interior characteristic foliation of a disk is reduced
(and the boundary is Legendrian in NAF or is transversal) it exhibits a
Legendrian tree with elliptic singularities as vertices and hyperbolic-containing
edges. To see this note that
the only interior singularities are positive elliptic and
negative hyperbolic, and the latter have stable separatrices coming from the
former thus making up the hyperbolic-containing edges of our graph. In fact
the graph is a tree, as discussed in \cite{[E2]}.
Indeed, the unstable separatrices
of these interior hyperbolic points must go to the boundary which allows one
to show that the graph is a deformation retract of the whole disk and
is thus connected and simply-connected.

Also important for dealing with these trees is
the following observation: each positive elliptic interior point (vertex of the
tree) must be connected to {\it at least one} positive hyperbolic point
on the boundary. This follows from the tightness of $(M,\xi)$,
because if a vertex has no connection to boundary hyperbolic points
then all the negative interior hyperbolic
points to which it is connected\footnote{(whose
separatrices form the edges adjacent to that vertex)} will have unstable
separatrices flowing to boundary elliptic points in pairwise fashion (i.e.
two hyperbolic separatrices flowing to each of these boundary
elliptic points). With the help of the elliptic pivot lemma, we would  thus
obtain a closed Legendrian curve consisting only of negative elliptic
points (from boundary) and negative hyperbolic points (from interior).
This would violate the tightness of $(M,\xi)$. Indeed, using the elimination
lemma (Lemma~\ref{lem-elim}) we could pairwise eliminate all singularities
and thereby create a limit cycle.
 Figure~\ref{tight}
illustrates this for the case of a vertex with only one attached edge.

\begin{figure}
 \centerline{\epsfxsize=10cm\epsfbox{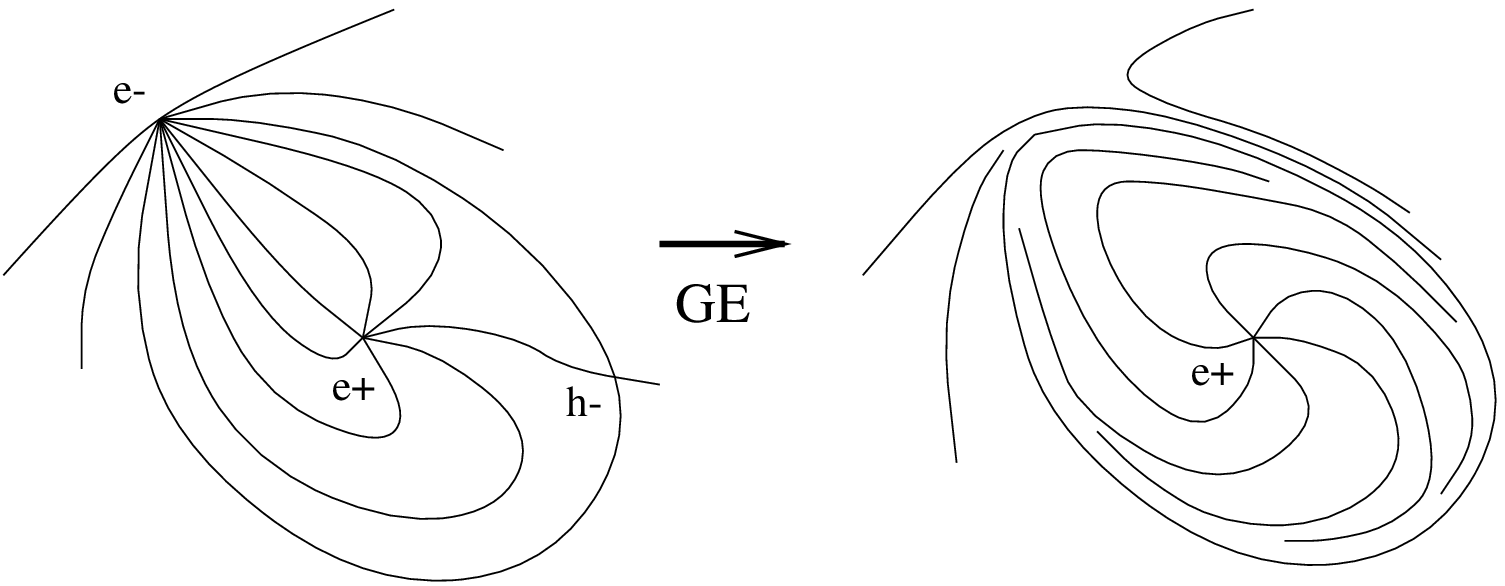}}\medskip
\caption{Tightness of $(M,\xi)$}
\centerline{$\Rightarrow$ Each $e+$ inside attached to $h+$ on $L$}
\label{tight}
\end{figure}

\subsubsection{Elliptic form on the spanning disk}\label{ellForm}
Let $D$ be a spanning
disk for the Legendrian unknot $L$. The characteristic foliation $D_\xi$ will be said to
be in {\it elliptic form} when the signs of boundary singularities alternate, all interior
singularities are elliptic positive or elliptic negative and, besides the direct connection
to its neighbor via a subsegment of $L$, each boundary point is connected only to
interior points, and moreover each (elliptic) interior point is connected
to at least 2 boundary hyperbolic points.\footnote{This last condition
is included to rule out the pathological situation with only one
interior elliptic point
and two boundary points (one elliptic, one hyperbolic); when
we have more than one interior elliptic point, it is implied by
the other conditions, and so redundant.} An example of an elliptic form
disk foliation is given in Figure~\ref{fig-simpleEFD}.
\begin{figure}
\centerline{\epsfxsize=10cm\epsfbox{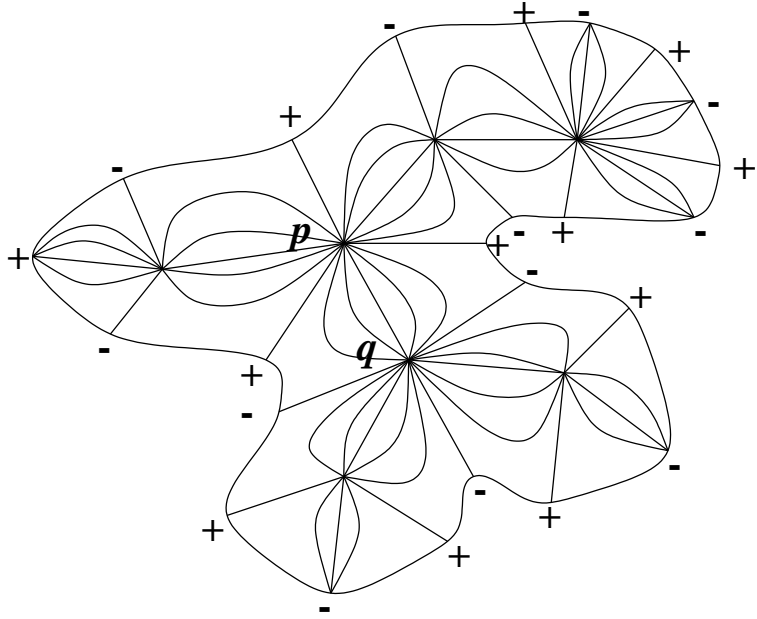}}\medskip
\caption{Example of an elliptic form disk foliation}
\label{fig-simpleEFD}
\end{figure}

These conditions imply that boundary hyperbolic points
of a given sign are connected to
each other in groups of
two or more via  their separatrices  which meet
in an elliptic interior point of that same sign. These separatrices thus divide
the surface into
regions\footnote{The word {\it region}
will be informally used in this paper to refer to a topological embedding
of the disk minus some segments of the boundary, such that the boundary is
piecewise smooth.} of
type(a) or type(b) on which there are no other interior
singularities (see Figure~\ref{fig-regionTypes}).
This is illustrated in
Figure~\ref{fig-flowSpaces}
for several vertices of a spanning disk.
\begin{figure}
\centerline{\epsfxsize=10cm\epsfbox{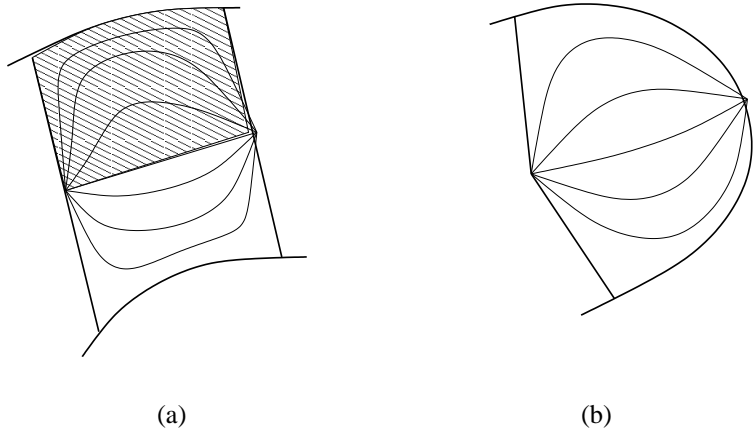}}\medskip
\caption{Types of region on disk in elliptic form}
\label{fig-regionTypes}\end{figure}
When we consider just half of a type(a) region, as shaded in
Figure~\ref{fig-regionTypes}, we will call it a semi-type(a) region.
Note, on the other hand, that if a disk can be divided into type(a) and
type(b) regions, then it automatically satisfies the properties of elliptic
form. So for an embedded disk $D$, the existence of a
decomposition of $D_\xi$ into type(a) and type(b) regions is equivalent to
$D_\xi$ being in elliptic form.
\begin{figure}[h!]
\centerline{\epsfxsize=10cm\epsfbox{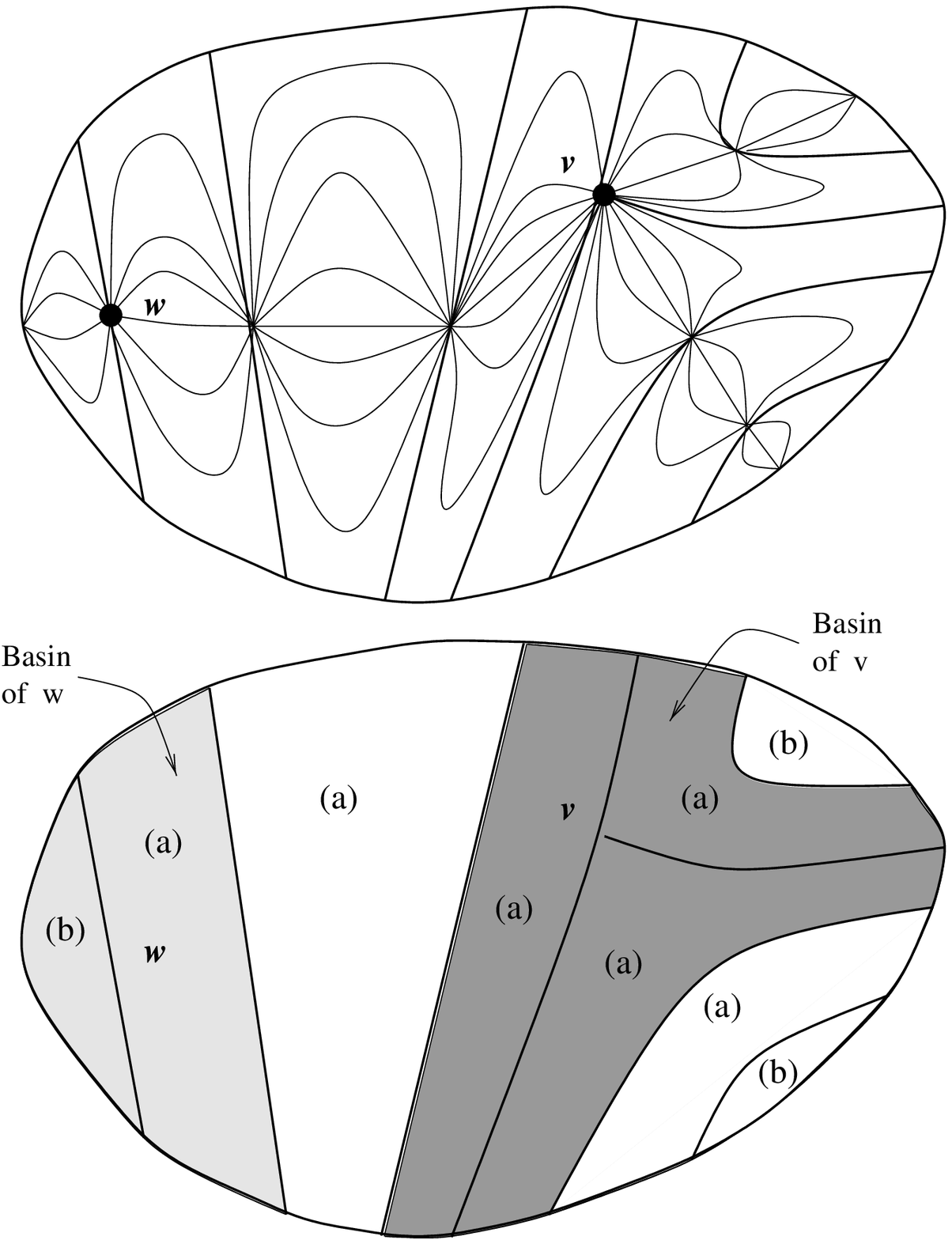}}\medskip
\caption{Example of how an elliptic form disk}
\centerline{decomposes into type(a) and type(b) regions}
\label{fig-flowSpaces}\end{figure}

It is also useful to note that if a disk with elliptic form foliation
is oppositely oriented, it is still in elliptic form; i.e. the definition is
invariant under reversal of sign of the singularities.

\subsubsection{Final step of standardization:
reduced form $\longrightarrow$ elliptic form}

The following lemma shows we can realize step 3. of the standardization
program for disk foliation, namely we can achieve elliptic form.

\begin{Lem} In tight $(M,\xi)$,
let D be an oriented embedded disk with Legendrian boundary $L$. Suppose
$D_\xi$ is in reduced form with normal absorbing boundary. Then there is a $C^0$-small
isotopy of
$D$ fixing $L$ which results in a new spanning disk having characteristic foliation in
elliptic form.\label{lem-ellForm}
\end{Lem}

{\it Proof.\ \ \ }
We are assuming $D_\xi$ is reduced so it has a Legendrian tree $T$.
Consider the singularities along $L$. Every positive singularity is
hyperbolic, and so connected to only one interior singularity.
The negative singularities, on the other hand, are elliptic and may be
connected to several negative interior singularities (which are hyperbolic).
We want to alter the foliation so that each boundary singularity is
attached to only one interior singularity of its same sign.

The strategy for accomplishing this will be to
use the positive boundary singularities' unique
connections to interior points in order to divide the disk into regions
that are easier to deal with.
Let $P \subset D_\xi$ consist of the stable separatrices of positive
boundary hyperbolic points, together with those points, and all positive
interior elliptic points. This divides $D$ into regions: the connected
components of $D\setminus P$. Every region
meets $L$ in a disjoint union of open Legendrian segments
which contain
exactly one singularity (negative elliptic). Those regions which
meet $L$ in a single segment are actually of type(b)\footnote{(see
Figure~\ref{fig-regionTypes})}, having a positive elliptic point on the
interior and a negative elliptic point on the boundary. These we will
not alter. Suppose, on the other hand, that $R$ is a region meeting
$L$ in more than one segment, i.e. that
$R\cap L$ has $n$ components for $n\geq 2$.
These are Legendrian segments.
Note that on the interior of $R$ there are $n-1$ negative
hyperbolic points. Each of these is attached to exactly 2 of the segments
just mentioned. Convert all $n-1$ negative elliptic points
on these segments to hyperbolic points (Lemma~\ref{lem-convert}),
and then cancel negative pairs (Lemma~\ref{lem-elim}). One negative interior
elliptic point remains, and no other interior points. Now consider
$\partial R \setminus L$. It
has $n$ connected
components (since $R\cap L$ did), and each component is either a
single Legendrian segment from $B$ or a union of 2 such segments.
In the first case, the segment is homotopic to a point relative to $L$
and consists of a single hyperbolic separatrix between a positive elliptic
point $p$ on the interior and a positive hyperbolic point on the boundary.
Convert the boundary point (Lemma~\ref{lem-convert}) to an elliptic point
and then cancel the newly created interior hyperbolic point with $p$
(Lemma~\ref{lem-elim}). This is illustrated for a typical region $R$
in Figure~\ref{fig-getNegCluster}.
Note that every time we carry out this procedure,
we create a type(b) region, having a negative elliptic point on the
boundary. All the rest of $R$ is broken into type(a) regions by the hyperbolic
separatrices that connect interior to boundary points. So $R$ can be
decomposed into type(a) and type(b) regions.

\begin{figure}
\centerline{\epsfxsize=14cm \epsfbox{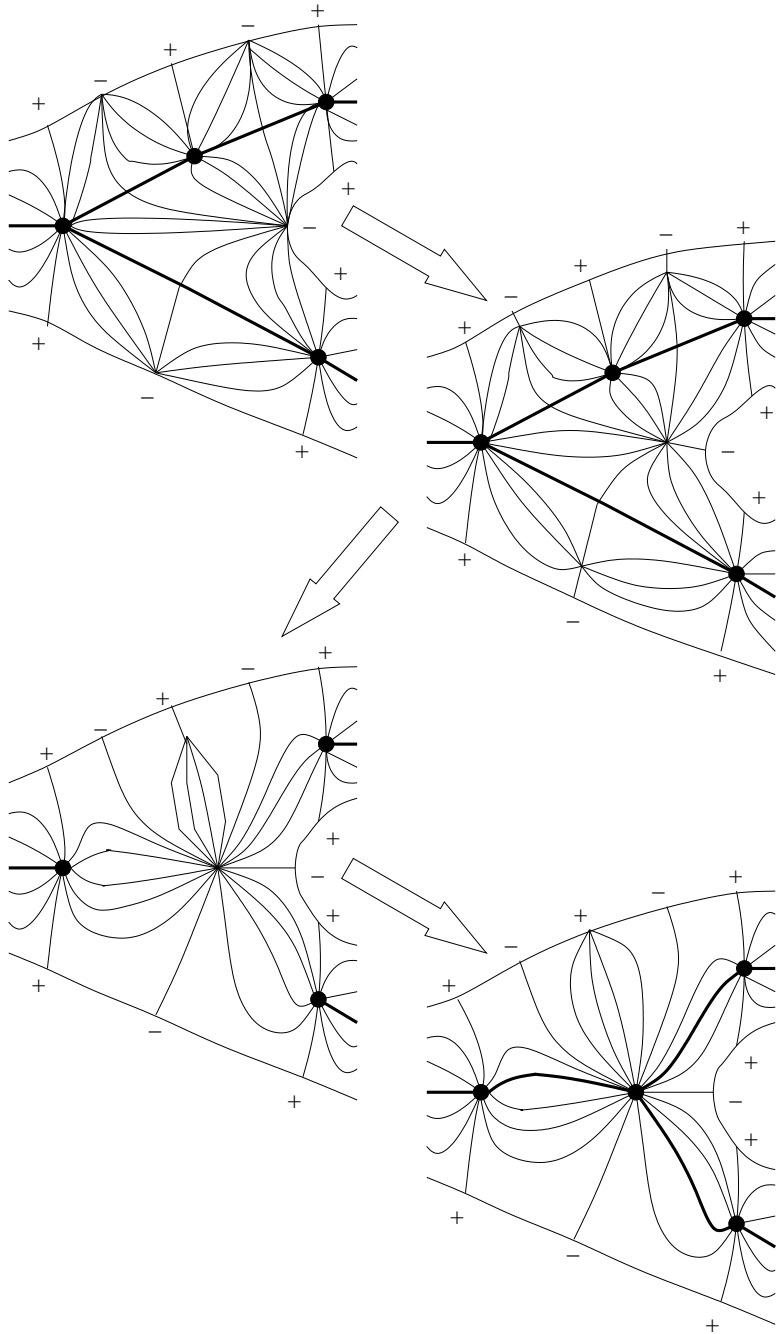}}\medskip
\caption{Example of procedure to alter region $R$}
\vfill\hbox{}
\label{fig-getNegCluster}
\end{figure}

Assume we have performed the above procedure on all regions determined by $P$.
Each is then decomposable into type(a) and type(b) regions,
so this is true of the whole disk; i.e. the disk foliation
is now in elliptic form.
\qed

\subsubsection{Legendrian trees of an elliptic form disk: skeletons and extended
 skeletons}\label{sec-skel}
Just as reduced form with normal absorbing boundary
results in an obvious Legendrian tree
on the interior, so
too does elliptic form. Suppose $D$ is an embedded disk with Legendrian
boundary $L$ whose characteristic foliation $D_\xi$ is in elliptic form.
Consider all (interior as well as boundary) elliptic singularities of $D_\xi$
and for each pair of these points which are connected by a smooth family of
singularity-free Legendrian arcs, choose one representative of this family of
arcs.
Define the {\it skeleton} of $D_\xi$ to
be the Legendrian graph that has for vertices all interior elliptic points
and for edges the representative arcs between these points.
Let the {\it extended
skeleton} be the Legendrian graph containing the skeleton but having as new
vertices the elliptic boundary points, and as new edges the representative
arcs attaching these points to vertices of the skeleton.
This is illustrated in Figure~\ref{fig-exskel} below.
We have defined the skeleton and extended skeleton as abstract graphs,
as well as giving their corresponding Legendrian embedding.
Note that as Legendrian graphs, they are well-defined by $D_\xi$
up to self-diffeomorphism of $D_\xi$ and choice of representative
arcs. As abstract graphs, they are thus well-defined up to graph isomorphism
given the topology of $D_\xi$.

\begin{figure}
\centerline{\epsfxsize=12cm\epsfbox{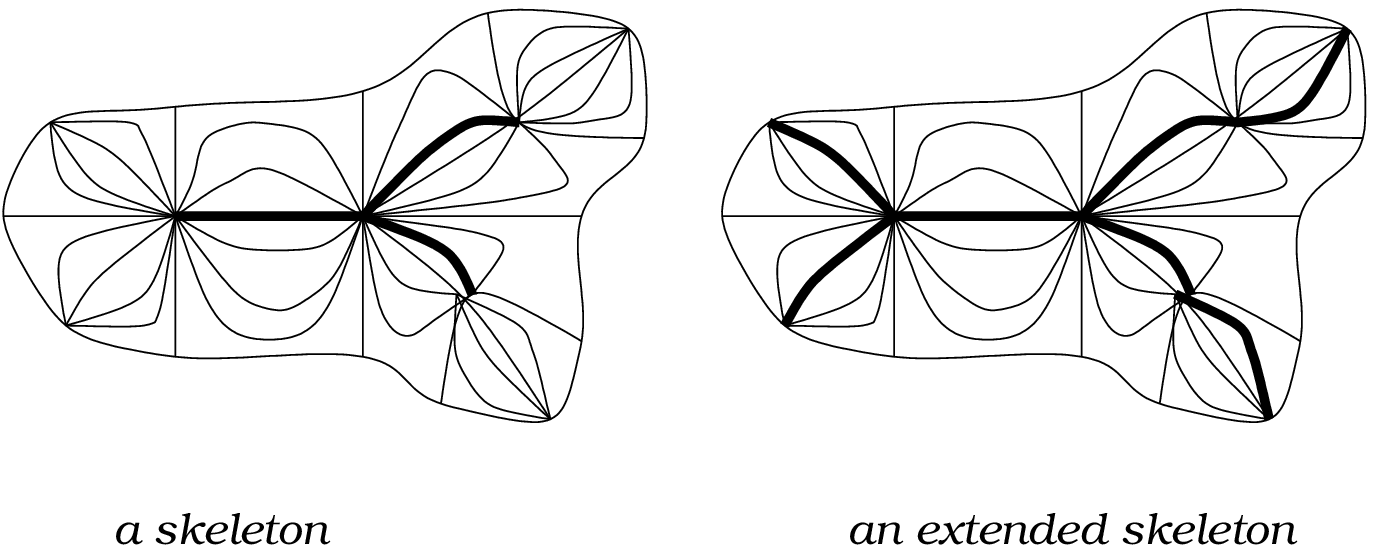}}\medskip
\caption{Extending the skeleton}
\label{fig-exskel}
\end{figure}

Moreover, the
skeleton and extended skeleton are both deformation retracts of the
disk. This follows from the decomposition of an elliptic form disk into
type(a) and type(b) regions as described in
Section~\ref{ellForm},
because on each such region, we can clearly define an appropriate deformation
retract which is standard where such regions meet. As a result, we see that
the skeleton and extended skeleton are in fact connected and simply-connected.
They are thus Legendrian trees.

\subsubsection{Graph terminology}\label{sec-twig}

In dealing with graphs, we will use the following terminology. The
{\it valence} of a vertex in an abstract graph will be defined as the
number of edges attached to that vertex. When a vertex of a tree has
valence one it
will be called an {\it end vertex}. The edge to which it is attached
will be called an {\it end edge}.
Note that a connected graph is a tree if
and only if all vertices have valence less than or equal to two.
Edges which share a common vertex are said to be
{\it adjacent edges}. As well, vertices which are the endpoints of some edge,
are called {\it adjacent vertices}.


\subsubsection{Signed tree exhibited by an elliptic form disk}

We will say the embedding $\mathcal T$ of a tree in the plane is {\it signed}
if there is a map from the set of vertices of  $\mathcal T$ to  $\{+,-\}$
such that adjacent vertices have opposite sign. Such a  map will be called
a {\it signing} of  $\mathcal T$.

Suppose  $\mathcal T$ and $\mathcal T'$ are signed embeddings of abstract trees in the
plane.
We will say they are
{\it equivalent} if there is a homeomorphism of the plane which
takes one tree to the other and respects the signings.

Recall that the skeleton of an elliptic form disk $D$ was defined as
a Legendrian embedding $\phi$
of an abstract tree $T$ into $D$. By using an embedding $\varphi$
such that $D$ is the image under $\varphi$ of the standard unit disk,
we obtain an embedding $\varphi^{-1}\circ\phi$ of $T$ into the plane.
It has a natural signing: that given by the signs of singularities
in $D_\xi$.
Observe, that the equivalence class of this signed embedded tree is well
defined by the topology of $D_\xi$, and does not depend
on the choice of $\varphi$. The same is true for the
extended skeleton.
Given a signed embedding of a tree in the plane, we will say it is
{\it exhibited} by an elliptic form disk $D$ if it is in the equivalence class
determined by the {\it extended skeleton} of $D_\xi$, and we will call
this equivalence class the {\it extended skeleton type} of $D$.

\subsubsection{Acceptable tree embeddings}

Let $\mathcal T$ be a planar embedding of an abstract tree.
We will say  $\mathcal T$ is {\it acceptable} if it
has the following properties:
\begin{enumerate}
\item it has at least one edge,
\item all edges are straight segments with slope between $\pm \epsilon$,
\item at any given vertex, there is at most one edge attached
on the left side of that vertex, all others are attached on the right
(this implies there is a {\it left-most vertex}),
\item the left-most vertex is an end vertex
(this implies there is a {\it left-most edge}).
\end{enumerate}

\section{Proof of the Main Theorem} \label{sec:main}

\subsection{General scheme of proof}
We now outline our strategy for proving the
Main Theorem. Overall, we will show any two Legendrian
unknots with a given value of $(tb,r)$ are isotopic to a lift of
the (unique) catalog front with that value (and so are isotopic to each
other).
This will be broken into the following steps.

\begin{enumerate}
\item construct from
any signed acceptable tree embedding $\mathcal T$ a
wavefront $\mathcal W_T$ with the property that
its Legendrian lift has an elliptic form  spanning disk
exhibiting $\mathcal T$ (Sections~\ref{sec:algo}, ~\ref{sec:treebased},
~\ref{sec:ess}, ~\ref{sec:frntSpDisk}).

\item establish a transversal homotopy from any wavefront
obtained by the above construction\footnote{These will be called
{\it tree-based} wavefronts.} to the
(unique) catalog front with the same value of
$(tb,r)$ (Lemma~\ref{lem:homoToCat}).

\item show that all Legendrian knots having elliptic form spanning
disks exhibiting the same embedded tree are Legendrian isotopic;
i.e. the extended skeleton type of elliptic form spanning disk
determines the Legendrian isotopy class of the (Legendrian)
boundary (Lemma~\ref{lm:shrink}).
\end{enumerate}

It is possible to avoid the use of wavefronts by completely standardizing
the disk foliation (see \cite{[F1]}), but then a more general version
of step (2) is required to finish the proof; namely, that any two
Legendrian knots with diffeomorphic spanning-disk foliation (not necessarily
of elliptic form) are Legendrian isotopic. This is shown in \cite{[F1]} using
a 1-parametric version of results from \cite{[E3]}, which although formulated in \cite{[E3]}
is not proved there. Thus we have decided instead to follow the method of proof
outlined above.
The use of wavefronts in the present
proof is similar to, though more elaborate than, the use of wavefronts
in \cite{[E2]} (transverse knot version of our Theorem).

\subsection{Wavefront arguments}\label{sec-front}

\subsubsection{Front construction algorithm}\label{sec:algo}

Suppose we
are given a signed acceptable tree embedding $\mathcal T$. The algorithm below
constructs from $\mathcal T$ a wavefront that will be denoted
$\mathcal W_T$.

Choose a small neighborhood around each vertex, so that no neighborhoods
overlap.
In the diagrams below, the thick lines
represent edges, the thin curves represent fronts, and the dotted
boxes represent these vertex neighborhoods. For each edge, we will
refer to the subsegment between the two endpoint neighborhoods as
the {\it open edge}.
The algorithm below steps through all vertices,
setting $v$ equal to the current vertex and then defining a portion of
front corresponding to $v$'s neighborhood and any open edges attached to
$v$ on the right.
We assume these portions
of front are
always connected to each other smoothly.

\medskip
{Anchor step: }

\centerline{\parbox[t]{12cm}{
Let $v$ be the left-most vertex of $\mathcal T$.
If $v$ is positive use the first front below and if $v$ is negative use the
second front to define the portion of $\mathcal W_T$ corresponding to
$v$'s neighborhood and to the (single) attached open edge.
Now set $v$ to be the right endpoint of this edge, and {\it go to
the induction step}.
}}\medskip

\centerline{\epsfxsize=8cm\epsfbox{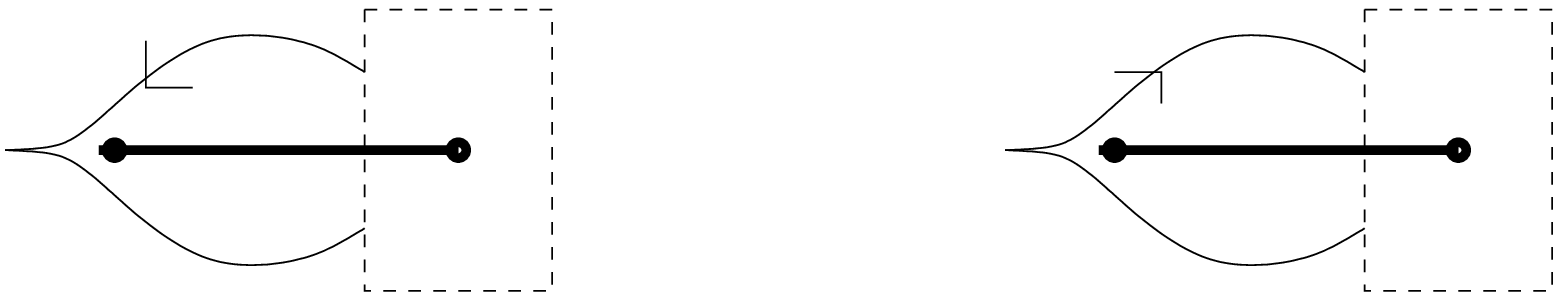}}

\medskip
{Induction step: }

\centerline{\parbox[t]{12cm}{
Let $n$ be the
number of edges attached to $v$ on the right side. We will
define the portion of $\mathcal W_T$ corresponding to $v$'s neighborhood and
to these $n$ open edges.
}}\medskip
\centerline{\parbox[t]{12cm}{
If $n=0$, then $v$ is an end-vertex.
Use the following front, and then {\it end the induction}.}}\medskip

\centerline{\epsfxsize=2cm\epsfbox{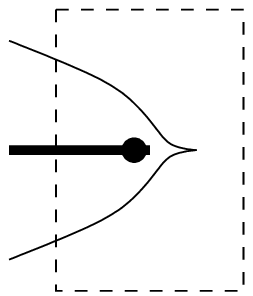}}

\centerline{\parbox[t]{12cm}{
If $n\neq 0$, then $v$ is not an end-vertex.
Replace the subtree that is attached to the right of $v$
with a new subtree that is obtained from the old
by reflecting in the horizontal axis.
Then, if $n=1$ use the front at left,
if $n>1$ use the front at right.
Now, number the $n$ edges attached to $v$ on the right
sequentially from $1..n$ starting at the top.
For each $i=1..n$, set $v$ to the right endpoint of the
$i$'th edge and {\it repeat the induction step} for this $v$.
}}

\centerline{\epsfxsize=12cm\epsfbox{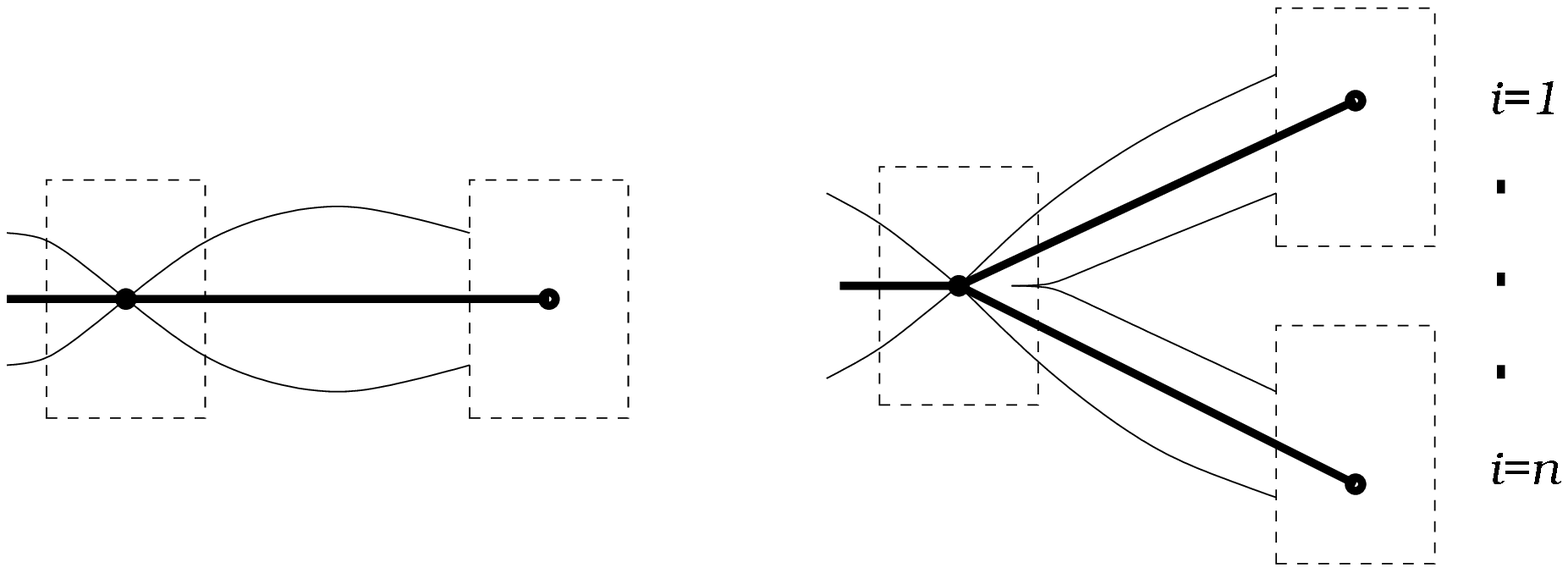}}

It is useful to note that as a result of this construction:
a cusp in $\mathcal W_T$ corresponding to a positive (end) vertex in $\mathcal T$
is oriented downward, as are both arcs of a self-intersection point
(i.e. front crossing)
corresponding to
a positive (non-end) vertex. The opposite is true for negative vertices.

\subsubsection{Tree-based wavefronts}\label{sec:treebased}

Define a {\it tree-based front} to be any front that is obtained as
$\mathcal W_T$ for some signed acceptable tree embedding $\mathcal T$.

Suppose $T$ is an abstract tree with the property that at most
one of its vertices
has valence greater than two, and at most one of the edges attached to this
vertex is a non-end edge. We will then say $T$ is
{\it almost linear}. Note that the class of catalog fronts is exactly
the class of tree-based fronts of the form $\mathcal W_T$,
such that $\mathcal T$ is a signed
acceptable embedding of an almost linear tree.

\subsubsection{Simplifying fronts}\label{sec:simplification}

A transversal homotopy of fronts lifts  to a Legendrian isotopy between the
corresponding Legendrian knots. Indeed,
this follows from the fact that we can canonically lift each transitional front in the
homotopy of fronts and
 this lift (a Legendrian knot) will be  embedded iff each self-intersection point of the
corresponding front is transversal (recall from Section 1.2).
\medskip

\begin{Lem} Any tree-based
front can be transversally homotoped to the catalog front
with that value of $(tb,r)$.
\label{lem:homoToCat}\end{Lem}

{\it Proof.\ \ \ } Consider two signed planar tree embeddings  $\mathcal T$
and $\mathcal T'$.
Suppose $\mathcal T'$ is obtained from $\mathcal T$ by moving an
end edge from one positive (resp. negative) vertex to another positive
(resp. negative) vertex. We claim
there is a transversal homotopy of
wavefronts taking $\mathcal W_T$ to $\mathcal W_{\mathcal T'}$. This proves the lemma,
since any tree $\mathcal T$ can be reduced to an almost linear tree by
a sequence of such moves\footnote{We establish (as stated) only movement of
an edge between {\it same sign} vertices, and with such a tool can only
reduce a signed tree to {\it almost} linearity - not necessarily linearity.}.
To verify the claim, consider the subtree
$\mathcal S$ which is the common part of $\mathcal T$ and $\mathcal T'$,
let $e$ be the edge appearing only in $\mathcal T$ and let $e'$ be the
edge only in $\mathcal T'$. Let us designate the two vertices $v$ and $w$
(in all three trees), such that $e$ is attached to $v$ and $e'$ is attached
to $w$. If neither of these is an end vertex in $\mathcal S$,
then note that $\mathcal W_T$ and
$\mathcal W_{T'}$ differ only by the placement of an upward (resp. downward)
oriented {\it zig-zag}\footnote{(see
Figure~\ref{fig-frntSimplif1})}.
If, on the other hand, either of $v$ and $w$ is an end vertex in
$\mathcal S$, then apply the
transversal homotopy of fronts shown in Figure~\ref{fig-frntSimplif1}
- near $v$ in  $\mathcal W_T$, near $w$ in
$\mathcal T'$ (start at right and go to either top left or bottom left).
\begin{figure}
\centerline{\epsfxsize=12cm\epsfbox{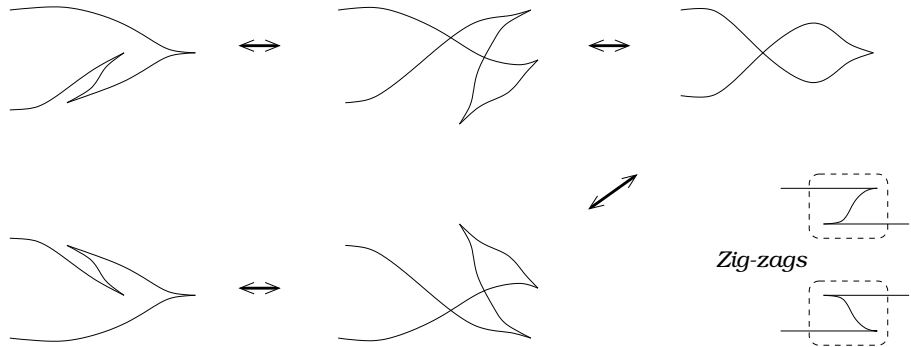}}
\caption{Transverse homotopy of fronts for proof of Lemma~\ref{lem:homoToCat}}
\label{fig-frntSimplif1}
\end{figure}
The two resulting fronts then differ only in the placement
 of an
   upward (resp. downward) oriented
zig-zag, (while $\mathcal W_S$ is identical to both
fronts with the zig-zag removed). So to prove the claim it suffices to
show that we can move an upward (resp. downward) zig-zag from any
position in a connected front to any other by a transversal front homotopy.
Clearly, we can displace the zig-zag past any transversal crossing in this
manner, so we need only check whether we can also displace it past cusps.
For an upward (resp. downward) cusp, this is clear. The same
transversal homotopy of fronts just given above
 shows we can do so for downward (resp. upward)
cusps as well (start at left, go to right and then back to left) .
\qed

\subsection{Spanning disks}

\subsubsection{Exceptional spanning disks and elliptic form disks}
\label{sec:ess}

Let $D$ be an embedded disk in $(M,\xi)$ with Legendrian boundary
$L$. If the characteristic foliation $D_\xi$ can be decomposed into
regions of the types shown in
Figure~\ref{fig-sector}, allowing each singularity curve to have
one interior corner (point of non-smoothness),
then we will say $D$ is an {\it exceptional
spanning disk} for $L$.
\begin{figure}
\centerline{\epsfxsize=12cm\epsfbox{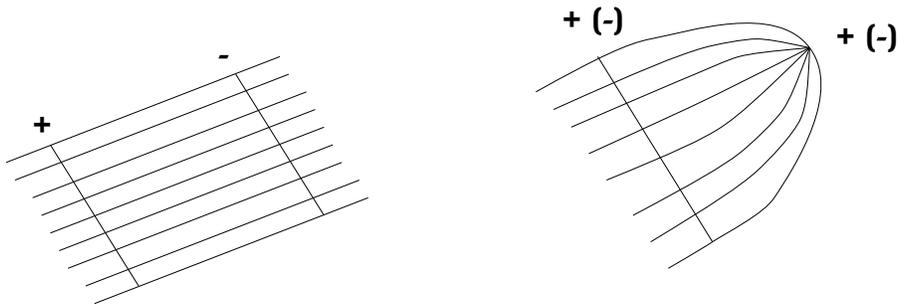}}\medskip
\caption{Types of region composing an exceptional spanning disk}
\label{fig-sector}
\end{figure}
Note that
if a spanning disk is in elliptic form then
we can convert it to an exceptional spanning disk,
by application of Lemma~\ref{lem-singCurve} (using the decomposition
into type(a) and type(b) sectors that we have for elliptic form foliations),
and vice versa. We will also, by analogy, speak of the {\it extended
skeleton exhibited by an exceptional spanning disk}.

\subsubsection{Exceptional spanning disk for the lift of $\mathcal W_{\mathcal T}$}
\label{sec:frntSpDisk}

Suppose we are
given a tree-based front $\mathcal W_{\mathcal T}$ with Legendrian lift $\mathcal L$.
We will define an exceptional spanning disk $D$
for $\mathcal L$ by specifying $D_\xi$, the family of Legendrian curves
that foliate $D$. This can be done by giving the corresponding
family of wavefronts for these Legendrian curves.
Since $\mathcal W_{\mathcal T}$ was defined
(in the front construction algorithm of section~\ref{sec:algo})
portion by portion for the edges of $\mathcal T$, we specify the relevant
wavefronts in this same manner, namely we do so for each type of
wavefront portion created in that algorithm.
This is given in Figure~\ref{fig-ess}.
\begin{figure}
\centerline{\epsfxsize=12cm\epsfbox{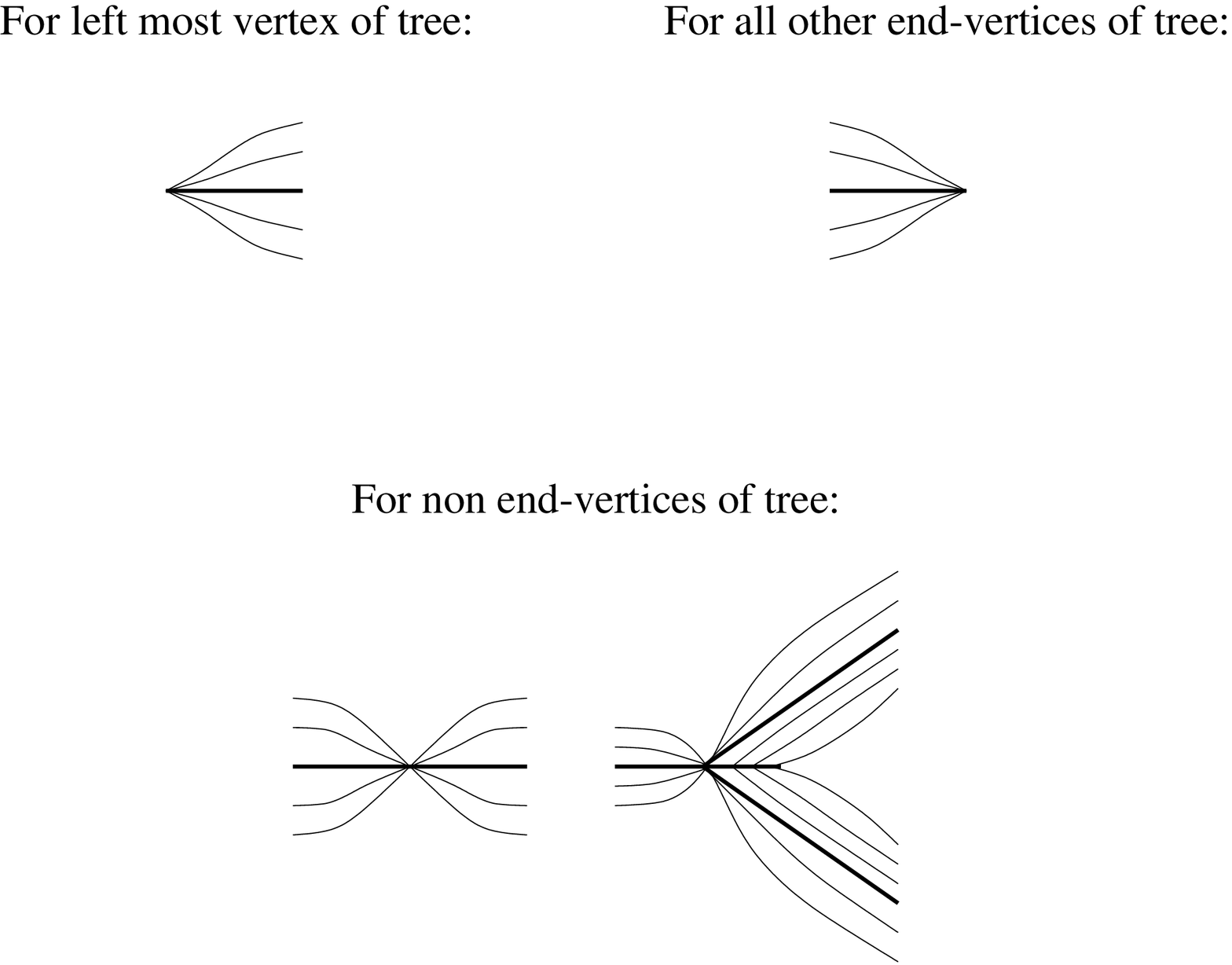}}\medskip
\caption{Construction of an exceptional
spanning disk $D$ for the lift of $\mathcal W_{\mathcal T}$
(by specifying wavefront projection of $D_\xi$ on each
portion of
$\mathcal W_{\mathcal T}$ created in the algorithm of Section~\ref{sec:algo})}
\label{fig-ess}
\end{figure}

\subsubsection{Spanning disks exhibiting the same extended skeleton}

\begin{Lem}\label{lm:shrink}
Suppose Legendrian knots $L$ and $L'$ bound  $D$ and $D'$  with diffeomorphic
characteristic
foliations in elliptic form. Then
$L$ and $L'$ are Legendrian isotopic.
\end{Lem}

{\it Proof.\ \ \ }
First, convert the elliptic form spanning disks to exceptional spanning disks
(as mentioned in section~\ref{sec:ess}).
Call these new disks also $D$ and $D'$.
Because the extended skeletons $T$ and $T'$
exhibited by $D$ and $D'$ are diffeomorphic
there exists, according to Lemma~\ref{lm:Leg-ambient} a
global contact isotopy taking
$T'$ to $T$. We may assume the surfaces coincide in a small
neighborhood $R$ of this common extended skeleton.
Written right on each surface,
there is an isotopy supported in the complement of small
neighborhoods of the end vertices which brings the portion of the
Legendrian knot outside these neighborhoods arbitrarily close to the
skeleton. This is indicated in Figure ~\ref{fig-isoAsWrit}.

Suppose that the relevant portion
of $L$ and also of $L'$ are each isotoped far enough that they reach the
region $R$ (around the common extended skeleton) where $D$ and $D'$
coincide.
By applying the Elliptic Pivot Lemma
(i.e. Lemma~\ref{lem-pivot}), we may extend each of
these isotopies
to the neighborhoods of end vertices. We thus obtain Legendrian
isotopies taking $L$ and $L'$ to the same Legendrian knot.
Thus $L$ and $L'$
are Legendrian isotopic.
\qed

\begin{figure}[h!]
\centerline{\epsfxsize=10cm\epsfbox{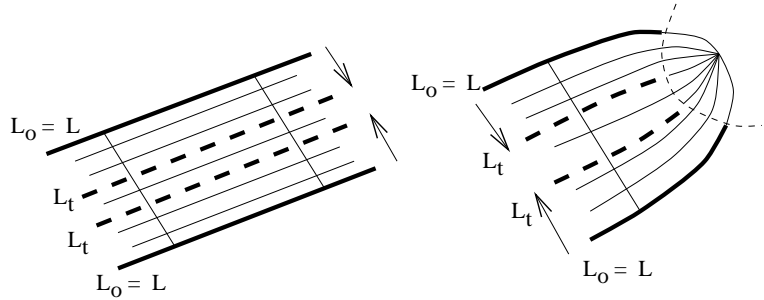}}\medskip
\caption{The isotopy written on an exceptional spanning disk}
\label{fig-isoAsWrit}
\end{figure}

\subsection{Proof of Main Theorem}\label{sec:mainProof}

{\it Proof.\ \ \ }(of Theorem~\ref{thm:main})
Given a Legendrian unknot $L$ in tight $(M,\xi)$, let $D$ be an
oriented embedded disk with $\partial D=L$. By applying Lemma~\ref{lem-stdize}
followed by Lemma~\ref{lem-ellForm}, we obtain a spanning disk $D'$ for $L$
that is in elliptic form. Let $\mathcal T$ be an acceptable signed planar tree
embedding
that is in the extended skeleton type of $D'$. Let $\mathcal W_T$ be the tree-based
front constructed from $\mathcal T$ by the algorithm of section~\ref{sec:algo}.
Let $\mathcal L$ be a Legendrian lift of  $\mathcal W_T$ to a Darboux neighborhood
in $(M,\xi)$. Then $L$ and $\mathcal L$ are Legendrian isotopic by
Lemma~\ref{lm:shrink} since $\mathcal L$ has an elliptic form spanning disk that
exhibits the same $\mathcal T$ as does the disk $D'$. Now, by
Lemma~\ref{lem:homoToCat}, all lifts like $\mathcal L$ having a
given value of $(tb,r)$ are Legendrian
isotopic to any lift of the (unique) catalog front with that value of $(tb,r)$.
So, any Legendrian unknot having the same value of $(tb,r)$ as $L$ has
will be Legendrian isotopic to $L$.

\qed

\section{Legendrian knots in overtwisted contact $3$-manifolds}\label{sec:twist}

Knots can be classified up to isotopy or up to global
diffeomorphism.
In the classical case  of knots in $\R^3$
these two problems are equivalent because the group of compactly supported diffeomorphisms of $\R^3$ is connected
(in fact even contractible, see \cite{[Hat]}). The same is true for Legendrian knots in tight
(i.e., standard) contact $\R^3$.
Indeed, according to \cite{[E3]} the group of compactly supported contact diffeomorphisms  of the standard contact $\R^3$
is connected   as well.  However, as it was observed by K. Dymara (see \cite{[D1],[D3]}
 and Corollary \ref{cor:Dymara} below), the group of coorientation preserving contactomorphisms of any closed overtwisted contact  manifold is disconnected.  
  
  We  consider in the next sections the problem of {\it coarse} classification of
Legendrian knots in an overtwisted contact manifold, i.e. the problem of classification of
Legendrian knots up to a global, coorientation preserving contact
diffeomorphism. The status of the Legendrian isotopy problem is discussed in Section
 \ref{sec:coarse-isotopy} below.

\subsection{Coarse classification of loose knots}
Let $(M,\xi)$ be an overtwisted contact manifold.
A Legendrian knot (or link) $L\subset M$ is called {\it loose} if
the restriction of the contact
structure $\xi$ to the complement $L^c=M\setminus L$ is still overtwisted.
Otherwise, i.e.
if $\xi|_{M\setminus L}$ is
tight, we will call $L$  {\it exceptional}.

One can immediately observe the following examples of loose knots.
\begin{Lem}\label{lem:someloose}
\begin{enumerate}
\item[a)]
Suppose that $(M,\xi)$ is  a non-compact manifold {\it overtwisted at
infinity}\ \footnote{
i.e. overtwisted    outside  of any compact set.
It is proven in \cite{[E6]} that $\R^3$, for instance, has a unique overtwisted at infinity contact structure.}
then any Legendrian link $L\subset M$ is loose.
\item[b)]  Let $(M,\xi)$ be any overtwisted contact manifold and $L$
a topologically trivial Legendrian knot with $tb(L)\leq 0$, then $L$ is loose.\footnote{More constraints on exceptional knots follow from 
 Theorem  \ref{thm:except} below.}
 \end{enumerate}
\end{Lem}

{\it Proof.}
Part (a) is clear. If $tb(L)=0$ then the image of $L$ under Reeb flow for a 
short time has linking number $0$ with $L$, so it spans an overtwisted
 disk in $L^c$. 
Now suppose $tb(L)<0$. Then $L$ has a spanning disk $D$ with NAF boundary.
If $L^c$ were tight then
we could standardize the disk interior too, and thus obtain standard foliation
on the whole closed disk. This would imply
a neighborhood $B$ of $D$ is isomorphic to a neighborhood of a standard
disk (as listed in the catalog), and hence is tight. But $D^c$ is also a tight ball,
so then all of $M$ would be tight.
\qed
\medskip

\begin{Rem}{\rm In our earlier paper  \cite{[EF]} we   claimed without a proof a stronger version of 
\ref{lem:someloose}, that any homological to 0 Legendrian knot in any contact manifold is loose if it violates the Bennequin inequality \ref{thm:ineq}. This is wrong even for topologically trivial knots, as  our Theorem~\ref{thm:except}  below
shows.}
\end{Rem}

The problem of coarse classification of loose Legendrian knots is of pure homotopical nature:
\begin{Prop}\label{prop:loose-knots}
Let $L_1,L_2\subset (M,\xi)$ be two loose Legendrian knots. Suppose that there exists a diffeomorphism
$f:M\to M$ which sends $L_1$ to $L_2$, $\xi_1|_{L_1}$ to $\xi_2|_{L_2}$ and such that the plane fields
$\xi_2$ and $df(\xi_1)$ are homotopic on $M\setminus L_2$ relative to the boundary. Then $L_1$ and $L_2$
are coarsely equivalent.
\end{Prop}
\begin{Cor}
Two topologically trivial  loose Legendrian knots in an overtwisted contact manifold $(M,\xi)$ are coarsely equivalent if and only if they have the same values
of $tb$ and $r$.
\end{Cor}
 \begin{Rem} \label{rem:rotchoice}
 {\rm If the Euler class $e(\xi)$ does not vanish, then, unlike in the tight case, the definition of $r(L)$ is ambiguous and may depend on a choice of a disk $D$ spanning the knot $L$, $ r(L)=r(L|D)$, see Section \ref{sec:tbr}. The above corollary should be understood in the sense that  if  two knots have the same $tb$, and there are spanning discs such that $r(L_1|D_1)=r(L_1|D_2)$ then there exists a coarse equivalence which sends $D_1$ onto $D_2$.  
 }
 \end{Rem}
\medskip
\begin{proof}[Proof of Proposition \ref{prop:loose-knots}]
By the Darboux-Weinstein theorem  (see Section \ref{sec:contact-basic} above) we can assume that the
 diffeomorphism $f$ sends $\xi_1$ to $\xi_2$ on a neighborhood $U_1\supset L_1$ and
 $U_2=f(U_1)\supset L_2$. 
 Then the contact structures $\xi_2$ and $\widetilde \xi_1=df(\xi_1)$ coincide on the boundary of $U_2$ and are homotopic as plane fields via a homotopy fixed on $U_2$.
 Hence, according to the classification of overtwisted contact structures in \cite{[E5]}  (see also Theorem
 \ref{thm:overtwisted} below)
 there exists an isotopy   $h_t: M \to M, t\in[0,1],$ 
 which is fixed on $U_2$ and such that $h_0=\id$, 
 $dh_1(\widetilde\xi_1)=\xi_2$. Then the composition $h_1\circ f$ is a contactomorphism $(M,\xi_1)\to (M,\xi_2)$ as required which sends $L_1$ to $L_2$.
\end{proof}

\noindent{\bf Remark.}  Fuchs and  Tabachnikov informed us that they independently observed  a similar result.
A related result in the context of Legendrian isotopy classification of loose knots is proved by K. Dymara in \cite{[D1]} under the additional assumption that there exists an overtwisted disk not meeting either $L_i, i=1, 2$, see Corollary  \ref{cor:Dymara}.2 below.
\medskip
 
\subsection{Coarse classification of exceptional knots} \label{sec:excep}

Topologically trivial {\it exceptional} knots in an overtwisted contact $S^3$ can be completely coarsely classified using the theorem
of Giroux-Honda which classifies tight contact structures on solid tori.
Let us recall that according to \cite{[E5]}  (positive) overtwisted contact structures on $S^3$ are classified by the homotopy class of the corresponding cooriented plane field, which is in turn defined by its
{\it Hopf invariant}.  Namely, let us fix a  trivialization of $TS^3$, say by the frame
$ip,jp,kp\in T_pS^3, p\in S^3$, where we view $S^3$ as the unit sphere in $\R^4$ identified with the quaternion space $\HH$.
Then any cooriented plane field gives a map $S^3\to S^2$, and we denote by $h(\xi)$ its Hopf invariant, i.e. the linking number of properly oriented pre-images of two regular points in $ S^2$.  In particular, the Hopf invariant of the standard contact structure  $\zeta$ orthogonal to the Hopf fibration is equal to $0$.    
   We will denote by $\xi_h$ the unique positive overtwisted contact structure with the Hopf invariant $h$.  
\footnote{
It is also  customary in the contact geometric  literature to use, instead of the Hopf invariant, the so-called 
$d_3$-invariant, introduced by R. Gompf  in \cite{[Go]}, see also \cite{[DGS]}. For $S^3$ the invariants $d_3$ and $h$ are related by the formula $d_3=-h-\frac12$.}

Let us recall that according to R.  Lutz (see \cite{[Lu]}), any plane  homotopy class  of contact structures can be obtained from a given one by Lutz twists.  
Let us compute the Hopf invariant of the overtwisted contact structure $\xi$ obtained from $\zeta$ by  the $\pi$-Lutz twist along one of the fibers, say $F=\{e^{it}, t\in\R/2\pi\Z\}$, of the Hopf fibration. Note that the linking number of two  Hopf fibers  is equal to $+1$.
Take  a trivialization $U=D^2\times S^1$ of the Hopf fibration near one of the fibers. Then the normal vector field
$jp, p\in F$ rotates  $-2$ times with respect to the constant framing given by the trivialization (this is why the transversal unknot represented by the fiber has self-linking number equal to $-1$).
 Suppose we perform the $\pi$-Lutz twist of the contact structure $\zeta$ along $F$. Given any unit  vector field ${\bf v}=
 a(ip)+b(jp)+c(kp)$, $p\in F$,  with constant coefficients in the basis $ip,jp,kp\in T_pS^3$ we denote by $\Gamma_{\bf v}$  the set of points in $S^3$ where the normal vector field to the contact structure points in the direction of ${\bf v}$. Then  $\Gamma_{\bf v}$  is the preimage of the point $(a,b,c)\in S^2\subset\R^3$ under the Gauss
map 
 $S^3\to S^2$ which corresponds to the contact structure. Clearly,  if $b^2+c^2\neq 0$ then  $\Gamma_{\bf v}$ is a circle spiraling around $F$ 
 minus $2$ times, exactly as does the normal vector field
$jp, p\in F$.  The Hopf invariant  $h(\xi)$ is by definition  the linking number of  $ \lk (\Gamma_{\bf v},\Gamma_{\bf v'})$ for two different vector fields
${\bf v}$ and ${\bf v}'$. Note that by a continuity argument all these spirals should be coherently oriented. On the other hand,        $\Gamma_{\bf v'}$  is isotopic to $F$  in the complement of    $\Gamma_{\bf v}$, and hence 
$$h(\xi)= \lk (\Gamma_{\bf v},\Gamma_{\bf v'})= \lk (\Gamma_{\bf v},F)=1-2=-1,$$ i.e. $\xi=\xi_{-1}$.
\medskip

Note that given a $k$-component link $L=L_1\cup\dots L_k\subset S^3$ of transversal  knots in $\zeta$ with selflinking numbers
$l_i$, $i=1,\dots, k$,   simultaneous $\pi$-Lutz twists along all components of $L$ produce an overtwisted contact structure $\xi$ with
$$h(\xi) =\sum\limits_1^k l_i+2\sum\limits_{1\leq i<j\leq k}  l_{ij},$$ where $l_{ij}=\lk(L_i,L_j)$ is the linking number of positively oriented transversal curves $L_i$ and $L_j$. In particular, one can observe
\footnote{This was pointed out to us by E.Giroux.} that  the Lutz twist along  $k$ fibers of the Hopf fibration produces a contact structure with the Hopf invariant $k(k-2)$. 

\begin{Thm} \label{thm:except}
The contact manifold $(S^3, \xi_h)$ admits an exceptional unknot if and only if
$h=-1$. Moreover,
exceptional unknots are classified up to coarse equivalence by the invariants
$tb$ and $r$ and the following is a complete list of equivalence classes:
$(tb,r)=(1,0)$, $(tb,r)=(n,\pm(n-1))$ for positive integer $n$.
\end{Thm}

This confirms Conjecture 41 of [EtN], which predicted   $\xi_{-1}$ to be the
only contact structure on $S^3$ for which exceptional unknots exist.
 It also implies a corrected version of  Conjecture 42 of [EtN]. This conjecture claimed
 a reduced set of possible values of $(tb,r)$. 
 John Etnyre informed us that jointly with T. Vogel they independently proved Theorem \ref{thm:except}.
 
 \medskip
 Before proving Theorem \ref{thm:except} we need to  develop some necessary preliminary information and recall some known facts.
  
\medskip 

\noindent{\bf 1. Neighborhoods of  Legendrian knots.}
 Let $\xi$  be a contact structure in a contact manifold $(M,\xi)$, and $L\subset M$   a Legendrian knot.  
Let $(x,y,z)$, $y,z\in\R$, 
$x\in\R/2\pi\Z$,  be the canonical Darboux coordinates in a neighborhood of $L$, so that the contact structure 
$\xi$ on this neighborhood is given by the  contact $1$-form 
 $dz-ydx$ and $L$ is given by equations $y=z=0$.  Passing to cylindrical coordinates
 $$(x,r,\theta)\mapsto (x,r\cos\theta,r\sin\theta), \;\;x,\theta\in\R/2\pi\Z, r\in[0,\infty).$$
 we get
 $$dz-ydx=r\cos\theta d\theta+\sin\theta dr -r\cos\theta dx,$$ and   $L=\{r=0\}$.
 Note that the characteristic foliation on the torus $T_\eps=\{r=\eps\}$ is given 
 by the form $\cos\theta(d\theta-dx)$, and thus $T_\eps$ is ruled by Legendrian curves $\theta=x+\const$, and has two {\it singularity curves} $\theta=\pm\frac{\pi}2$, where the contact structure is tangent to the torus (see  Section \ref{singCurves} above for the definition of singularity curves). \footnote{Note that $T_\eps$ is convex in the sense of \cite{[EG]}, i.e. admits a transversal contact vector field (e.g. the field $Y=z\frac{\p}{\p z}+y\frac{\p}{\p y}$). The {\it dividing curves} of a convex surface, defined by Giroux (see \cite{[Gi2]}), are the sets of points where the contact vector field is tangent to the contact plane field. They are, up to isotopy, independent of a choice of the contact vector field (and for $Y$ are given 
 by the equation $\theta=0,\pi$). Hence the number of dividing curves and their slopes are respectively the same as 
those of the singularity curves.}

 We say that a torus $T$ with a fixed coordinate  system 
 $x,\theta\in\R/2\pi\Z$ has a {\it Legendrian ruling with the slope $\lambda$} if it has a foliation by Legendrian curves
 which is isotopic to the linear foliation with the slope $\lambda$. For a rational number $\lambda=\frac pq$ the torus
  $T$ is ruled by closed  Legendrian curves in the class of the curve $qx=p\theta$.
  In particular, the torus $T_\eps$ is ruled by Legendrian curves with the  slope $1$.
  
 \begin{Lem}\label{lm:shape}
 \begin{enumerate}
\item  For any function  $\phi:\R/2\pi\to(0,\infty)$ such that $\phi'(\pm\frac{\pi}2)=0$
the torus $T_\phi=\{r=\phi(\theta)\}$  has  two singularity curves $\theta=\pm\frac{\pi}2$ and admits a Legendrian foliation with a slope $\lambda=\lambda(\phi)$. 
\item
 For any   number $\mu$ there exists a function $\phi:\R/2\pi\to(0,\infty)$, arbitrarily $C^0$-close to $\eps$  which satisfies $\phi'(\pm\frac{\pi}2)=0$ and such that  $\lambda(\phi)=\mu$. 
 \end{enumerate}
 \end{Lem}
 \begin{proof}
 The characteristic foliation defined      in coordinates $(\theta,x)$ by the equation
 \begin{equation}\label{eq:char-fol}
( \phi\cos\theta+\phi'\sin\theta) d\theta-\phi\cos\theta dx=0, 
 \end{equation} 
 or equivalently
$$ \frac {dx}{d\theta}=1+\psi'\tan\theta,$$
where $\psi(\theta)=\ln\phi(\theta)$. The condition  $\phi'(\pm\frac{\pi}2)=0$ ensures that  
 $T_\phi$ has two singularity curves $\theta=\pm\frac{\pi}2$ and that the Legendrian foliation transversely intersect these curves. To prove the second part of the lemma
 it remains to find a periodic $\psi$, which is $C^0$-close to the constant $\ln\eps$ which satisfies the condition
\begin{equation}\label{eq:int}
\frac1{2\pi}\int\limits_0^{2\pi} (1+\psi'\tan\theta)d\theta=\mu.
\end{equation}
Let us consider the function $\alpha_{\delta,\sigma}(u)=\delta\left(-1+\left(\frac u\sigma\right)^4\right)$. Then $-\delta\leq\alpha_{\delta,\sigma}(u)\leq 0 $ for $u\in[-\sigma,\sigma]$ and 
$\alpha_{\delta,\sigma}(\pm\sigma)=0$.
On the other hand $$\int\limits_{-\sigma}^\sigma\frac{\alpha'_{\delta,\sigma}}u du=\frac{8\delta}{3\sigma}\mathop{\to}\limits_{\sigma\to 0}\infty.$$
Note that for small $u$ we have $\tan(\frac\pi2+u)\simeq -u$. Take now a continuous piecewise smooth $2\pi$-periodic  function
\begin{equation*}
\psi(u)=
\begin{cases}\ln\eps+ \alpha_{\delta,\sigma_1}(u-\frac\pi2),& u\in[\frac\pi2-\sigma_1,\frac\pi2+\sigma_1],\cr
\ln\eps- \alpha_{\delta,\sigma_2}(u-\frac{3\pi}2),& u\in[\frac{3\pi}2-\sigma_2,\frac{3\pi}2+\sigma_2],\cr
 \ln\eps,&\hbox{elsewhere},
 \end{cases}
  \end{equation*}
 and smooth its corners. Then choosing appropriate sufficiently small $\sigma_1,\sigma_2$ and $\delta$ we can arrange that the integral \eqref{eq:int} takes an arbitrary value. while   the function $\psi$  
 is $C^0$-close to $\ln\eps$ and satisfies the condition
  $\phi'(\pm\frac{\pi}2)=0$. 
 \end{proof}
 Note that the Legendrian foliation on $T_\phi$ is transversal to the vector field $\partial_x$, and  hence we can orient it by the coordinate $\theta$. We will call this orientation {\it canonical}.
 
 \begin{Lem}\label{lm:ruling}
Let $\phi$ be a function as in Lemma \ref{lm:shape}.1. Consider a vector field $X$ tangent to the Legendrian foliation
on $T_\phi$, and which defines  the canonical orientation of $T_\phi$. Then $X$ extends to the solid torus $U_\phi=\{r\leq\phi\}$ as a non-vanishing vector field tangent to the contact structure $\xi$ and  the core Legendrian curve $L$, and defines the  given orientation of $L$.
\end{Lem}
 
 \begin{proof}
 For a positive number $\sigma<\min\phi$ we denote $\phi_t:=\sigma+t(\phi-\sigma)$. Then  tori $T_{\phi_t}$ foliate
 the domain $U_{\sigma,\phi}=\{\sigma\leq r\leq\phi\}$. Each of these tori admits a Legendrian foliation as in Lemma  \ref{lm:shape}.1.
 Hence the vector field $X$ extends to $U_{\sigma,\phi}$ as the vector field tangent to the  Legendrian ruling on each of
 the  tori $T_{\phi_t}$. It further extends to the solid torus $U_\sigma=\{r\leq\sigma\}$ as the vector field tangent to the Legendrian ruling on round tori $T_r$, $0\leq r\leq\sigma$. 
 \end{proof}

For what follows we will need  only    tori with rational slopes of the form  $\lambda= -\frac1n$, $n>0$, and will denote by $U_n$ the neighborhood
$\{r\leq\phi(\theta)\}$ with $\lambda(\phi)-\frac 1n$.
Any two such neighborhoods are related by a contact isotopy which fixes $L$. \footnote{It is interesting to notice that the tori with different slopes $\mu_1$ and $\mu_2$ are contactomorphic (via a Dehn twist along the singularity curve) if and only if $\mu_1-\mu_2\in\Z$, In particular, the boundary tori  $\p U_n$ and $\p U_m$ are not contactomorphic if $m\neq n$.}

\medskip
Consider a solid torus $Q=D^2\times S^1$ with the fixed coordinates $t,u\in\R/2\pi\Z$  the boundary torus $\p D^2\times S^1$, where $t\in\p D^2$ and $u\in S^1$.
We  say that a contact structure $\xi$ on  $N$ is in the {\it standard form } near $T$ with the {\it boundary slope} $\frac pq\neq 0$ if it is ruled by Legendrian curves in the class of the meridian, and has two singularity curves with the slope $\frac pq$.  The contact structure $\xi$ near $T$ can be defined by a contact
form $\lambda$ such that $\lambda|_T=\cos(u-\frac pqt)du$.  

\begin{Lem}\label{lm:complement}
Let $L$ be a topologically trivial Legendrian knot in $S^3$ with a contact structure $\xi$. Suppose $tb(L)=n$. 
Let $U_n$ be a standard    neighborhood defined above. 
\begin{enumerate}
\item Then the complementary torus $N^c=S^3\setminus \Int U_n$ is in the standard form near $T=\p N^c=\p U_n$, and the singular curves on its boundary have the slope $n$ with respect to the coordinate system given by the meridians of the solid tori $N^c$ and $U_n$.
\item Let $L'$ be any of Legendrian curves which form the Legendrian ruling. Let us orient $L'$ so it represents the same homology class of $U_n$ as $L$.
Then $r(L')=-r(L)$.
\end{enumerate}
\end{Lem}
\begin{proof}
1.  Let us consider the lattice $\Lambda=\Z\oplus\Z$ on the plane $\R^2$ with coordinates $(\theta,x)$.
By construction the boundary torus  $T=\p U_n$  is identified with $\R^2/2\pi\Lambda$. Let $e_\theta=(1,0),
e_x=(0,1)$ be the  vectors of the basis.
The vector $e_\theta$ corresponds to the meridian of $U_n$, while $e_x$ directs the singularity curve.
Let
 $\mu$ denote a vector corresponding to the meridian of the complementary torus $N^c$. 
We choose $\mu$ in such a way that the pair $(e_\theta,\mu)$ be a positive basis of the lattice.
 Hence $e_\theta\wedge\mu=1$. On the other hand, by the  definition of the Thurston-Bennequin invariant we must have
 $e_x\wedge \mu=n$. Hence, $\mu$ is the vector $-ne_\theta+e_x$, which is tangent to the ruling direction, as required.
 Vectors $\mu, e_\tau$ define a positive basis with respect to the orientation of $T$ as the boundary of $N^c$, and we have $e_x=\mu+
 n e_\theta$, i.e. the singularity curve has the slope $n$.

\medskip\noindent 2.  First note that the  standard orientation of the ruling is opposite to  the homological orientation as in the statement of the lemma.  
  Let us recall that the rotation number $r(L)$ is the obstruction for extending the tangent vector field to $L$ as a non-vanishing vector field tangent to $\xi$ along a surface spanning $L$ in $S^3$. By assumption, $L'$ spans a disc in the complement of $U_n$. On the other hand, the vector field tangent to $L$ and $L'$, but defining opposite homological   orientations of $L$ and $L'$, extends to $U_n$ as the vector field tangent to $\xi$. Hence, $r(L')=-r(L)$.
\end{proof}

\medskip
\noindent{\bf 2. Special contact structure on the solid torus.}
We  explicitly describe here special tight  contact structures on the solid torus  which we will need in the proof of Theorem \ref{thm:except}. 
 
Given $n \in \mathbb{N}$, let ${\mathcal F}_n$ be a characteristic foliation on the disk $D=D^2$ such that $\partial D$ is Legendrian
with $2n$ hyperbolic singularities of alternating signs and there are one central negative elliptic singularity and $n$ positive elliptic singularities only on the disk away from the boundary (see Figure~\ref{fig-sym}). Moreover 
assume that positive hyperbolic singularities on $\partial D$ are sources of the foliation
along $\partial D$ and that polar coordinates $(\rho, \alpha)$ are given on the disk
such that the foliation has rotational symmetry group $\Z_n$ and only the $2n$ Legendrian curves drawn radially, attaching boundary hyperbolic points to the center, are 
actually ever radial.
Note this disk may be realized as  a disk with Legendrian boundary 
having $(tb, r) = (-n, n-1)$.   \footnote{Alternatively the tightness of the germ of the contact structure inducing this characteristic foliation  follows  from Giroux's  criterion: that the 
dividing multicurve of the surface have
no closed closed components, see \cite{[Gi2]}.}  In fact, its characteristic foliation is  up to orientation the only possible $\Z_n$ rotationally symmetric foliation of a disk with $tb = -n$ 
boundary that can occur in a tight manifold.
This disk  is convex, according to
\cite{[Gi2]}, i.e. it admits a transversal contact vector field. It then follows that there exists a vertically invariant tight contact structure 
$\tilde \xi$ on $Z = D\times\R$ such that $D_{\tilde\xi} = {\mathcal F}_n$.
 Note that we must have $\partial_t \in \tilde\xi_p$ at some point $p$ along each
Legendrian arc between opposite sign boundary hyperbolic points. In constructing $\tilde\xi$ we may begin with a germ of contact structure on $D$ such that
this verticality of the contact structure occurs exactly once on each such arc. This means the characteristic foliation on 
$\partial Z$ has singularity curves over these points and no other singularities. We may further assume the germ of contact structure on $D$ to have $\Z_n$ rotational symmetry,
so that all of $\tilde\xi$ does as well.

\begin{figure}[h]
\centerline{\epsfxsize=8cm\epsfbox{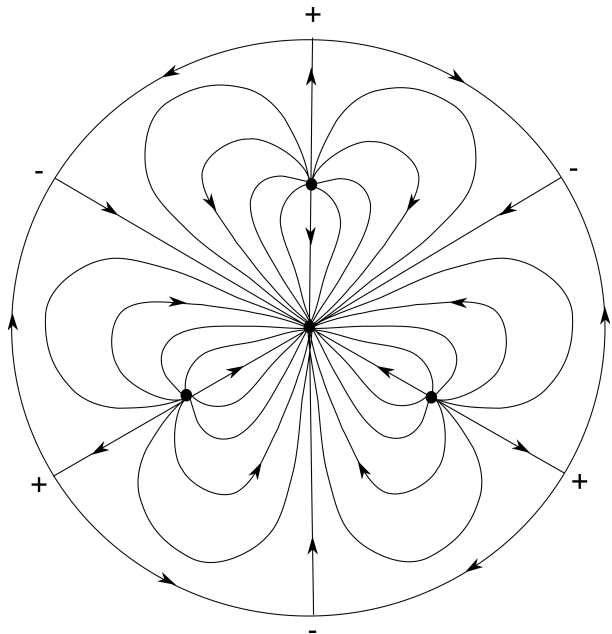}}\medskip
\caption{$\Z_n$ rotationally symmetric characteristic foliation ${\mathcal F}_n$ of a disk $D$, for the case $n=3$. 
}
\label{fig-sym}
\end{figure}

Now let $\phi$
be the composition of translation in the vertical direction by $1$ unit and rotation of the disk
by $\frac {2\pi} n$. Note that $\tilde\xi$ is invariant under $\phi$. Let $\zeta_n^-$ 
be the natural contact structure induced by $\tilde\xi$ on the quotient $Q = Z/_\phi$. Then $(Q, \zeta^-_n)$  is tight since it is covered by $(Z, \tilde\xi)$. Note that the foliation on its boundary still has exactly two singularity curves. 

We can do a similar construction, reversing the orientation of the disc $D$, i.e. beginning with the disc with one central positive elliptic singularity and $n$ negative elliptic singularities, and thus with $(tb,r)=(n,1-n)$.
We denote that resulting tight contact structure on $Q$ by $\zeta_n^+$.

In either case, we denote by $\bt$ the  contact vector field  in $Q$ which is the image of the vertical vector field $\partial_t$ under the quotient map $D\times \R\to Q$.
 
\medskip
\noindent{\bf 3. Giroux-Honda's classification of tight contact structures on the solid torus.}
The following theorem is an extract from  Giroux-Honda's classification of tight contact structures on the solid torus (see \cite{[Gi3]} and \cite{[Ho1]}).  
\begin{Thm}[Giroux, Honda]\label{thm:Giroux-Honda}
\begin{enumerate}
\item Let $\xi$ and $\xi'$ be   tight contact structure on  a solid torus $N$ which are in the standard form near the boundary with the boundary  slope $k\neq 0$.
Suppose that the Legendrian meridian $\mu$ has the same rotation number for both contact structures. Then $\xi$ and $\xi'$ are diffeomorphic by a diffeomorphism fixed on the boundary $T=\p N$.
\item Let $\xi$  be  a   tight contact structure on  a solid torus $N$ which is  in the standard form near the boundary with integer boundary  slope $n\neq 0$. Then the rotation  number $r(\mu)$ of the Legendrian meridian on its boundary is equal to $\pm (n-1)$. Moreover, both  these values of $r$ (the unique value $0$ if $n=1$) are realizable by tight contact structures.
\end{enumerate}
\end{Thm}
Note that the contact structures whose existence is claimed in \ref{thm:Giroux-Honda}.2 must by \ref{thm:Giroux-Honda}.1 be the contact structures $\zeta_n^\pm$ explicitely constructed in step 4 above.
\medskip 

\noindent{\bf 4. Proof of Theorem \ref{thm:except}.}
Let $L$ be an exceptional Legendrian knot in a contact  $(S^3,\xi)$.
Then, according to Lemma \ref{lem:someloose}.2 we have $tb(L)=n>0$.  Let $r = r(L)$. Consider the standard neighborhood $U_n\supset L$. Then,  according to  Lemma \ref{lm:complement}.1 the boundary $T=\p U_n$ is ruled by Legendrian curves in the class of the meridian of the complementary torus $N^c=S^3\setminus \Int U_n$, and using Lemma \ref{lm:complement}.2  we conclude that the rotation number of any Legendrian meridian is equal to $-r(L)=-r$. 
Hence, Giroux-Honda's Theorem \ref{thm:Giroux-Honda} implies that $r=\pm(n-1)$ and that there exists a contactomorphism  $(Q,\zeta_n^\pm)\to (N^c,\xi|_{N^c})$, where the sign  is the {\it same} as the sign of $r$.
It remains to compute the Hopf invariant of the constructed contact structure
$$(S^3,\wt\xi_n^\pm)=(U_n,\xi|_{U_n})\mathop{\cup}\limits_f(Q,\zeta_n^\pm).$$ We do below the computation for the case of $r=1-n$. The case of the positive rotation number is similar.\footnote{In fact, the positive rotation case formally follows from the negative: if $L$ is a Legendrian unknot with $(tb,r)=(n,1-n)$, then the same knot with the opposite orientation has $(tb,r)=(n,n-1)$.}

 Let us choose a reference  vector field  $\bv\in TS^3$ as follows.
 On $U_n$ we take $\bv$ to be equal to the vector field $X\in\xi$ constructed in Lemma \ref{lm:ruling}.
  Note that on the boundary $\p Q=\p D\times S^1$, we have  $X=-\mu$  (we continue to use the notation introduced in the proof of Lemma \ref{lm:complement}.1), i.e. $X$
  is  tangent to the meridians $\p D\times x$, $x\in S^1$, of 
 $Q$, but determines their opposite orientation. Next, we extend $\bv$ to a small neighborhood  $\Omega=\{1-\sigma\leq \rho\leq 1\}$ of
  the boundary $T=\p Q=\{\rho=1\}$, where $\rho\in[0,1]$ is the radial coordinate on the disk $D$. We  extend $\bv$ to $\Omega$ as tangent to tori $\{\rho=\const\}$  and rotating  clockwise (with respect to the   orientation of   $T=\p Q$)
  from $-\mu$ to the vector field $e_x$ directing the singularity curves, 
  as $\rho$ decreases from $1$ to $1-\sigma$. Note that the vector field $e_x$ on $T$ coincides with the contact vector field $\bt$ constructed in Step 2 on Q. Hence, we can extend
   $\bv$ to the rest of $Q$ as  equal to $\bt$ elsewhere. 
  It is straightforward to check that the vector field $\bv$ thus defined is homotopic to the 
    basic vector field $ip\in TS^3$ of the chosen framing
  $(ip,jp,kp)$ of $TS^3$.
 
 Choose now a Riemannian metric on $S^3$ in the following special way. Let us recall that   $Q$ is the quotient of $D\times \R$ by a map $\phi:D\times \R$ which is the composition of translation by $1$ and rotation by $\frac{2\pi}n$. The map $\phi$ is an isometry of the standard Euclidean product metric on $D\times S^1$, and hence $Q$ inherits the quotient metric. We extend this metric arbitrarily to the complement $U_n=S^3\setminus Q$. Let us denote by $\bw$ the vector field normal to the contact plane field $\wt\xi$.
 By definition the Hopf invariant $h(\wt\xi)$ is the linking number of appropriately oriented curves
 $\Gamma_\pm=\{\bw=\pm\bv\}$.
 First, note that $\Gamma_\pm\cap U_n=\varnothing$. Indeed, in $U_n$ the contact vector field  $\bv$ is tangent to $\xi$ while $\bw$ is orthogonal to $\xi$. We also have $\Gamma_\pm\cap \Omega=\varnothing$. Indeed, let us lift all the objects to the universal cover $D^2\times S^1$.   The vector field $bt$ is vertical,  the background metric
 is the Euclidean product metric, and the contact structure, as well as  the  lifted vector  fields  $\bv$ and $\bw$ are invariant with respect to translations along the vertical axis. The only possible points of $\Omega$  where we could have   $\bv=\pm\bw$ are along separatrices, which connect radially (i.e. in the $\partial_\rho$-direction) to hyperbolic boundary points.
  Indeed,  only along the separatrices  both vector fields are tangent to concentric tori $\{\rho=\const\}$. However, as $\rho$ decreases from $1$ to $1-\sigma$, both vector fields rotate clockwise, $\bv$ rotates $-\frac\pi2$ from  $-\mu$ to $e_x$, while  
 $\bw$ rotates by a small angle away from $\pm e_x$. In both cases we cannot have $\bv=\pm\bw$ anywhere in $\Omega$. Finally, inside $Q\setminus\Omega$  the curves
 $\Gamma_+$  and $\Gamma_-$ coincide with the locus of  respectively positive or negative elliptic singular points of  characteristic foliations on discs $D\times x,x\in S^1$. Hence, $\Gamma_+$ is the core circle of $Q$, while $\Gamma_-$ is isotopic in the complement of $\Gamma_+$ to  the core circle $L$ of  $U_n$.  
  The analysis of the Hopf map near an elliptic point, shows that the orientation induced by the Hopf map on
   $\Gamma_\pm$ is the same as defined by the coorientation of the contact structure.  In other words, $\Gamma_+$ is oriented by $\bv$, while $\Gamma_-$ is oriented by $-\bv$, Therefore,   $\lk(\Gamma_+,\Gamma_-)=\lk(-L,\Gamma_+)=-1$.  \qed

\bigskip
 The exceptional knots $K_n^\pm$ with $(tb,r)=(n,\pm(n-1))$ in $\xi_{-1}$  can be explicitly exhibited in the contact  $(S^3,\xi_{-1})$, see \cite{[D1]}.
 Namely, let us view $S^3$ as a quotient  space of the solid torus $D^2\times S^1$ where  each longitude  $x\times S^1$ $x\in \p D^2$ on its boundary is  collapsed to a point. Consider a  contact structure $\xi$ given in cylindrical coordinates on $T$ by a 1=form $\cos f(r)dz+\sin f(r)d\theta$, where the function $f(r)$ has the following properties:
 $f$ is monotone function with $f(0)=f'(0)=0, f'(1)=0$ and $f'>0$ on $(0,1)$. Then, if $f(1)=\frac\pi 2$ then the contact structure is tight, and if $f(1)=\frac{3\pi}2$  the contact structure $\xi$ is isomorphic to $\xi_{-1}$.
 Consider the latter case.
 There are exactly two values $r_0,r_1\in(0,1)$, $r_0<r_1$ such that $\tan f(r_0)=-n, \tan f(r_1)=-\frac 1n$.
 Then the tori $T_{r_0}=\{r=r_0\}$ and   $T_{r_1}=\{r=r_1\}$ are foliated by Legendrian knots with $tb=n$ and $r=\pm(n-1)$, where the sign of the rotation number depends on the orientation of the knots.
 K. Dymara explicitely verified in \cite{[D1]} that these knots are exceptional. She also conjectured that two  Legendrian  exceptional Legendrian knots with the same $(tb,r)$, one on $T_{r_0}$ and the other on $T_{r_1}$  are not Legendrian isotopic. By Theorem~\ref{thm:except} they are, of course, coarsely equivalent.
 
\subsection{Coarse classification vs. Legendrian isotopy}\label{sec:coarse-isotopy} 
\subsubsection{Legendrian knots with negative $tb$}
\begin{Prop}\label{prop:coarse-isotopy}
Let $L_1,L_2$ be two topologically trivial Legendrian knots in an overtwisted contact manifold $(M,\xi)$ with the same values of $tb, r$. Suppose that  $tb(L_1)(=tb(L_2))< 0$. Then $L_1$ and $L_2$ are Legendrian isotopic.
\end{Prop}
\begin{proof}
According to Lemma \ref{lem:someloose}(a) the knots $L_1$ and $L_2$ are loose.
There exists a  Legendrian knot $L_0$ with the same $tb$ and $r$ which is contained
in a small Darboux ball. Proposition \ref{prop:loose-knots} then implies that all three knots, $L_0, L_1$ and  $ L_2$ are coarsely equivalent. 
Hence, $L_1$ and $L_2$ are contained in neighborhoods contactomorphic to the standard tight contact $3$-ball.   The space of embeddings of the tight $3$-ball in any contact manifold is connected, and hence the claim follows from Theorem \ref{thm:main} in the tight case.
\end{proof}
 
\subsubsection{Topology of the contactomorphism group of an overtwisted contact manifold}

Let $(M,\xi)$ be a connected  contact manifold. $M$ can be either closed, or non-compact. In the latter case all diffeomorphisms we consider are assumed to be with compact support. Similarly homotopies of contact structures and plane fields are always assumed to be  fixed at infinity.
Let  us denote by
\begin{itemize}
\item[--] $\Diff_0(M)$  the identity component of the group of compactly supported diffeomorphisms;
\item[--] $\Diff_0(M,\xi)$
the subgroup of $\Diff_0(M)$ consisting of those diffeomorphisms preserving the contact structure $\xi$ and its co-orientation;
 
\item[--] $\Distr(M|\xi)$ and  $\Cont(M|\xi)$ the spaces of respectively plane fields and contact structures  on  $M$ which coincide with $\xi$ at infinity if $M$ is non-compact;
\item[--] $\Distr_0(M|\xi)$ and $\Cont_0(M|\xi)$ the connected components of $\xi$ in $\Distr(M|\xi)$ 
and $\Cont(M|\xi)$, respectively. 
\end{itemize}

We recall: \nopagebreak 
\begin{Thm}[See \cite{[E5]}]\label{thm:overtwisted}
\begin{enumerate}
\item
If $\xi$ is overtwisted, then the group  $\Diff_0(M)$ acts transitively on $\Cont_0(M,\xi)$. 
\item
If $M$ is non-compact and $\xi$ is overtwisted at infinity then the inclusion $$j:\Cont_0(M|\xi)\hookrightarrow\Distr_0(M|\xi)$$ is a homotopy equivalence.
\end{enumerate}
\end{Thm}
\begin{Rem}\label{rem:warning}
 {\rm
Though the homotopy equivalence   is   stated in \cite{[E5]},   the proof there is given only for 1-parametric families, which only implies an isomorphism on $\pi_0$ and an epimorphism on $\pi_1$.}
\end{Rem}

 Thus, for any overtwisted $\xi$ the evaluation map $f\mapsto f_*\xi$ defines a Serre fibration 
 $\pi:\Diff_0(M)\to \Cont_0(M|\xi)$, with the fiber $\Diff_0(M,\xi)$. If $M$ is not compact and  $\xi$ is overtwisted at infinity,
 \footnote{Note that a complement of an   overtwisted disc in its arbitrarily small neighborhood is overtwisted, hence a complement of an overtwisted disc in any manifold is overtwisted at infinity.} then the base $\Cont_0(M|\xi)$
 of this fibration is homotopy equivalent to $\Distr_0(M|\xi)$.
 
 \begin{Cor}[comp. Dymara,\cite{[D2]}]\label{cor:Dymara}
 \begin{enumerate}
 \item  The classifying space \linebreak $\mathrm{B}\Diff_0(\R^3,\xi )$ is homotopy equivalent to $\Cont_0(\R^3|\xi)\simeq
 \Distr_0(\R^3|\xi)\simeq \Map(S^3,S^2)$
 for any overtwisted at infinity contact structure $\xi$ on $\R^3$. Here 
 we denote by $\Map(S^3,S^2)$ the space of based maps $S^3\to S^2$. In particular, 
 $$\pi_0(\Diff_0(\R^3,\xi))=
 \pi_1(\mathrm{B}\Diff_0(\R^3,\xi))=\pi_4(S^2)=\Z_2\,.$$
 \item The coarse classification of Legendrian knots in $(\R^3,\xi)$ coincides with their classification up 
 to Legendrian isotopy.
\end{enumerate}
\end{Cor}
 \begin{proof} 
The statement \ref{cor:Dymara}.1 follows from  the fact
that  the homotopy classes of plane fields on a 3-manifold $M$  coincide with homotopy classes of maps $M\to S^2$, and from
Hatcher's theorem \cite{[Hat]}
that  the group $\Diff_0(\R^3)$ is contractible. 

 To prove  \ref{cor:Dymara}.2 let us consider a contactomorphism $f\in\Cont_0(\R_3,\xi)$ which maps one of two coarsely equivalent Legendrian knots $L_1$ and $L_2$ onto the other one.
 According to  \ref{cor:Dymara}.1 there exists exactly two connected components of $\Cont_0(\R_3,\xi)$.
 Suppose that $f$ does not belong to  the identity component.  Take a ball $B\subset \R^3\setminus \supp f$ such that the contact structure $\xi|_{\Int B}$ is overtwisted at infinity. Then again applying  \ref{cor:Dymara}.1
 we construct a contactomorphism $g\in\Diff_0(\Int B,\xi)$ which is not isotopic to the identity in 
 $\Diff_0(\Int B,\xi)$. Consider a diffeomorphism $\widetilde f\in \Diff_0(\R^3,\xi)$ which is equal to $f$ outside $B$, and equal to $g$ on $B$. Then $\widetilde f$ is a coarse equivalence between $L_1$ and $L_2$, as was $f$.
However, the $\Z_2$-invariant associated by  \ref{cor:Dymara}.1 with a contactomorphism  is additive for contactomorphisms with disjoint support, and hence it is trivial for $\widetilde f$.
 Thus,  $\widetilde f$  is isotopic to the identity inside the group $\Diff_0(\R^3,\xi)$, and in particular, $L_1$ and $L_2$ are Legendrian isotopic.
 \end{proof}
 
 \begin{Rem}{\rm
 \begin{enumerate}
 \item Using  \ref{cor:Dymara}.1 and some algebraic topology, one can show (see \cite{[D3]}) that 
  for any overtwisted contact manifold  $(M,\xi)$ the group $\Diff_0(M,\xi)$ is disconnected. 
  \item Yu Chekanov informed us that he proved  that for an overtwisted contact structure $\xi_n$ on $S^3$ with the Hopf invariant $n$  one has
    \begin{equation*}
  \pi_0(\Diff_0(S^3,\xi_n))=
  \begin{cases}
  \Z_2\oplus\Z_2&\hbox{if}\; n=-1;\cr
  \Z_2 &\hbox{otherwise}.
  \end{cases}
  \end{equation*}
  \end{enumerate}
  }
\end{Rem}


\end{document}